\documentclass[11pt,letterpaper]{article}
\usepackage{graphicx}
\usepackage{epsfig,subfigure}
\usepackage{amssymb}
\usepackage{amsthm}
\usepackage{amsfonts}
\usepackage{amsmath}
\usepackage{multicol,multirow}
\usepackage[margin=1in]{geometry}

\newtheorem{theorem}{Theorem}[section]
\newtheorem{lemma}[theorem]{Lemma}
\newtheorem{proposition}[theorem]{Proposition}

\newtheorem{corollary}[theorem]{Corollary}

\newcommand{\prob}{\mathbb{P}}
\newcommand{\expc}{\mathbb{E}}

\newcommand{\bfy}{\mathbf{Y}}
\newcommand{\bfx}{\mathbf{X}}

\newcommand{\bfe}{\boldsymbol \varepsilon}

\newcommand{\R}{\mathbb{R}}

\DeclareMathOperator*{\argmin}{\arg\!\min}

\title{High-Dimensional Change-Point Estimation: Combining Filtering with Convex Optimization}
\author{Yong Sheng Soh$^{\dag}$ and Venkat Chandrasekaran$^{\dag,\ddag}$ \thanks{Email: ysoh@caltech.edu, venkatc@caltech.edu} \vspace{0.25in} \\ $^\dag$ Department of Computing and Mathematical Sciences \\ $^\ddag$ Department of Electrical Engineering \\ California Institute of Technology \\ Pasadena, CA 91125}

\begin{document}

\maketitle

%


%
%


\begin{abstract}
We consider change-point estimation in a sequence of high-dimensional signals given noisy observations.  Classical approaches to this problem such as the filtered derivative method are useful for sequences of scalar-valued signals, but they have undesirable scaling behavior in the high-dimensional setting.  However, many high-dimensional signals encountered in practice frequently possess latent low-dimensional structure.  Motivated by this observation, we propose a technique for high-dimensional change-point estimation that combines the filtered derivative approach from previous work with convex optimization methods based on atomic norm regularization, which are useful for exploiting structure in high-dimensional data.  Our algorithm is applicable in online settings as it operates on small portions of the sequence of observations at a time, and it is well-suited to the high-dimensional setting both in terms of computational scalability and of statistical efficiency.  The main result of this paper shows that our method performs change-point estimation reliably as long as the product of the smallest-sized change (the Euclidean-norm-squared of the difference between signals at a change-point) and the smallest distance between change-points (number of time instances) is larger than a Gaussian width parameter that characterizes the low-dimensional complexity of the underlying signal sequence.
\end{abstract}


\textbf{Keywords}: High-dimensional time series; convex geometry; atomic norm thresholding; filtered derivative.




\section{Introduction}

Change-point estimation is the identification of abrupt changes or anomalies in a sequence of observations.  Such problems arise in numerous applications such as product quality control, data segmentation, network analysis, and financial modeling; an overview of the change-point estimation literature can be found in \cite{BasNik:93,ChoRobSie:71,PooHad:08,VeeBan:12}. As in other inferential tasks encountered in contemporary settings, a key challenge underlying many modern change-point estimation problems is the increasingly large dimensionality of the underlying sequence of signals -- that is, the signal at each location in the sequence is not scalar-valued but rather lies in a high-dimensional space.  This challenge leads both to computational difficulties as well as to complications with obtaining statistical consistency in settings in which one has access to a small number of observations (relative to the dimensionality of the space in which these observations live). 


A prominent family of methods for estimating the locations of change-points in a sequence of noisy scalar-valued observations is based on the \emph{filtered derivative} approach \cite{AntHus:94,BasNik:93,BenBas:84,Ber:00,BerFhiGui:11}. Broadly speaking, these procedures begin with an application of a low-pass filter to the sequence, followed by a computation of pairwise differences between successive elements, and finally the implementation of a thresholding step to estimate change-points. A large body of prior literature has analyzed the performance of this family of algorithms and its variants \cite{AntHus:94,Ber:00,BerFhiGui:11}. Unfortunately, as we describe in Section \ref{sec:mainresults}, the natural extension of this procedure to the high-dimensional setting leads to performance guarantees for reliable change-point estimation that require the underlying signal to remain unchanged for long portions of the sequence.  Such requirements tend to be unrealistic in applications such as financial modeling and network analysis in which rapid transitions in the underlying phenomena trigger frequent changes in the associated signal sequences.

\subsection{Our contributions}


To alleviate these difficulties, modern signal processing methods for high-dimensional data -- in a range of statistical inference tasks such as denoising \cite{BhaTanRec:13,Don:95,GavDon:14}, model selection \cite{CanPla:09,MeiBuh:06,Wai:09}, the estimation of large covariance matrices \cite{BicLev:08b,BicLev:08a}, and others \cite{CanRomTao:06,CRPW:12,Don:06,FoyMac:13,RecFazPar:10} -- recognize and exploit the observation that signals lying in high-dimensional spaces typically possess low-dimensional structure. For example, images frequently admit sparse representations in an appropriately transformed domain \cite{BruDonEla:09,CanRomTao:06} while covariance matrices are well-approximated as low-rank matrices in many settings. The exploitation of low-dimensional structure in solving problems such as denoising leads to consistency guarantees that depend on the intrinsic low-dimensional ``complexity'' of the data rather than on the ambient (large) dimension of the space in which they live.  A notable feature of several of these structure-exploiting procedures is that they are based on convex optimization methods, which can lead to tractable numerical algorithms for large-scale problems as well as to insightful statistical performance analyses.  Motivated by these ideas, we propose a new approach for change-point estimation in high dimensions by integrating a convex optimization step into the filtered derivative framework (see Section \ref{sec:mainresults}).  We prove that the resulting method provides reliable change-point estimation performance in high-dimensional settings, with guarantees that depend on the underlying low-dimensional structure in the sequence of observations rather than on their ambient dimension.

To illustrate our ideas and arguments concretely, we consider a setup in which we are given a sequence $\bfy[t] \in \R^p$ for $t = 1,\dots,n$ of observations of the form:
\begin{equation} \label{eq:gaussianmodel}
\bfy[t] = \bfx^{\star}[t] + \bfe[t].
\end{equation}
Here $\bfx^{\star}[t] \in \R^p$ is the underlying signal and the noise is normally distributed $\bfe[t] \sim \mathcal{N}(0,\sigma^2 I_{p \times p})$.  The signal sequence $\mathcal{X} := \{\bfx^{\star}[t]\}_{t=1}^n$ is assumed to be piecewise constant with respect to $t$.  The set of change-points is denoted by $\tau^{\star} \subset \{1,\dots,n\}$, i.e., $t \in \tau^{\star} \Leftrightarrow \bfx^{\star}[t] \neq \bfx^{\star}[t+1]$, and the objective is to estimate the set $\tau^{\star}$.  A central aspect of our setup is that each $\bfx^{\star}[t]$ is modeled as having an efficient representation as a linear combination of a small number of elements from a known set $\mathcal{A}$ of building blocks or atoms \cite{BhaTanRec:13,CanRec:09,CanRomTao:06,ChaJor:13,CRPW:12,Don:95,Don:06,Faz:02,RecFazPar:10,TBSR:13}. This notion of signal structure includes widely studied models in which signals are specified by sparse vectors and low-rank matrices. It also encompasses several others such as low-rank tensors, orthogonal matrices, and permutation matrices.  The convex optimization step in our approach exploits knowledge of the atomic set $\mathcal{A}$; specifically, the algorithm described in Section \ref{sec:algorithm} consists of a denoising operation in which the underlying signal is estimated from local averages of the sequence $\bfy[t]$ using a proximal operator based on the atomic norm associated to $\mathcal{A}$ \cite{BhaTanRec:13,CRPW:12,Don:95}.  The main technical result of this paper is that the method we propose in Section \ref{sec:mainresults} provides accurate estimates of the change-points $\tau^{\star}$ with high probability under the condition:
\begin{equation}
\label{eq:changepoint_sufficient_simple}
\Delta_{\min}^{2} T_{\min} \gtrsim \sigma^2 \{ \eta^2(\mathcal{X}) +\log n\},
\end{equation}
where $\Delta_{\min}$ denotes the size (in $\ell_2$-norm) of the smallest change among all change-points, $T_{\min}$ denotes the smallest interval between successive change-points, and $n$ is the number of observations.  The quantity $\eta(\mathcal{X})$ captures the low-dimensional complexity in the signal sequence $\mathcal{X}:= \{\bfx^{\star}[t]\}_{t=1}^n$ via a Gaussian distance/width characterization, and it appears in our result due to the incorporation of the convex optimization step.  In the high-dimensional setting, the parameter $\eta^2$ plays a crucial role as it reflects the underlying low-dimensional structure in the signal sequence of interest; as such it is usually much smaller than the ambient dimension $p$ (we quantify the comparisons in Section \ref{sec:background}).  Indeed, directly applying the filtered derivative method without incorporating a convex optimization step that exploits the signal structure would lead to weaker performance guarantees in terms of $p$ rather than $\eta^2$.


The performance guarantee \eqref{eq:changepoint_sufficient_simple} highlights a number of tradeoffs in high-dimensional change-point estimation that result from using our approach.  For example, the appearance of the term $\Delta_{\min}^{2} T_{\min}$ on the left hand side of \eqref{eq:changepoint_sufficient_simple} implies that it is possible to compensate for one of these quantities being small if the other one is suitably large.  Further, our algorithm also operates in a causal manner on small portions of the sequence at any given time rather than on the entire sequence simultaneously, and it is therefore useful in ``online'' settings.  This feature of our method combined with the the result \eqref{eq:changepoint_sufficient_simple} leads to a more subtle tradeoff between the computational efficiency of the approach and the number of observations $n$; specifically, our algorithm can be adapted to process larger datasets (e.g., settings in which observations are obtained via high-frequency sampling, leading to larger $n$) more efficiently without loss in statistical performance by employing a suitable form of convex relaxation based on the ideas discussed in \cite{ChaJor:13}.  We discuss these points in greater detail in Section \ref{sec:tradeoffs}.


\subsection{Related work}

A recent paper by Harchaoui and L\'evy-Leduc \cite{HarLev:10} is closest in spirit to ours (in terms of providing change-point estimation guarantees of the form \eqref{eq:changepoint_sufficient_simple}); they describe a convex programming method based on total-variation minimization to detect changes in sequences of scalar-valued signals.  In addition to the restriction to scalar-valued signals, the technique in \cite{HarLev:10} requires knowledge of the full sequence of observations in advance.  As a result it is not directly applicable in high-dimensional and online settings unlike our proposed approach.

High-dimensional change-point estimation has received much attention in recent years based on different types of extensions of the scalar case.  The diversity of these generalizations of the scalar setting is due to the wide range of applications in which change-point estimation problems arise, each with a unique set of considerations. For example, several papers \cite{ChoFry:14,EniHar:13} investigate high-dimensional change-point estimation in settings in which the changes only occur in a small subset of components.  Therefore, assumptions about low-dimensional structure are made with regards to the pattern of changes rather than in the signal itself at each time instance (as in our setup).  Xie et al. \cite{XieHuaWil:12} consider a high-dimensional change-point estimation problem in which the underlying signals are modeled as lying on a low-dimensional manifold; although this setup is similar to ours, their algorithmic approach is based on projections onto manifolds rather than on convex optimization, and the types of guarantees obtained in \cite{XieHuaWil:12} are qualitatively quite different in comparison to \eqref{eq:changepoint_sufficient_simple}. We also note recent work on high-dimensional change-point problems in which the authors study the impact of random projections on the performance of classical algorithms such as the cumulative-sum method \cite{AstKir:14}.

\subsection{Paper outline}


Section~\ref{sec:background} gives the relevant background on structured signals that are concisely represented with respect to sets of elementary atoms as well as the analytical tools that are used in the remainder of the paper.  In Section \ref{sec:mainresults} we describe our algorithm for high-dimensional change-point estimation, and we state the main recovery guarantee of the procedure.  In Section \ref{sec:tradeoffs} we discuss the tradeoffs that result from using our approach, and their utility in adapting our algorithm to address challenges beyond high-dimensionality that arise in applications involving change-point estimation.  We verify our theoretical results with numerical experiments on synthetic data in Section \ref{sec:numericalresults}, and we conclude with brief remarks and further directions in Section \ref{sec:conclusions}.  The proofs are given in the Appendix.


\section{Background on Structured Signal Models}\label{sec:background}

\subsection{Efficient representations with respect to atomic sets}

We outline a framework with roots in nonlinear approximation \cite{Bar:93,DevTem:96,Jon:92,Pis:81} that generalizes several types of low-dimensional models considered in the literature such as sparse vectors and low-rank matrices \cite{BhaTanRec:13,BicLev:08b,BicLev:08a,CanRomTao:06,CRPW:12,Don:06,MeiBuh:06,OymHas:12,RecFazPar:10}.

Let $\mathcal{A} \subseteq \mathbb{R}^p$ be a compact set that specifies a collection of atoms. We say that a signal $\bfx \in \mathbb{R}^p$ has a concise representation with respect to $\mathcal{A}$ if it admits a decomposition as a sum of a small number of atoms in $\mathcal{A}$, that is, we are able to write
\begin{equation} \label{eq:atomicnormstructure}
\bfx = \sum^{s}_{i=1} c_i \textbf{a}_i, \textbf{a}_i \in \mathcal{A}, c_i \geq 0,
\end{equation}
for some $s \ll p$.  Sparse vectors and low-rank matrices are examples of low-dimensional representations that are expressible in this framework.  Specifically, an atomic set for sparse vectors is the set of signed standard basis vectors $\mathcal{A} = \{\pm \textbf{e}_i | 1\leq i \leq p\}$, while a natural atomic set for low-rank matrices is set of all rank-one matrices with unit Euclidean norm.  Other examples include binary vectors (e.g., in knapsack problems \cite{ManRec:11}), permutation matrices (in ranking problems \cite{JagSha:11}), low-rank tensors \cite{KolBad:09}, and orthogonal matrices.  Such classes of signals that have concise representations with respect to general atomic sets were studied in the context of linear inverse problems \cite{CRPW:12}, and subsequently in the setting of statistical denoising \cite{BhaTanRec:13,ChaJor:13}.

In comparison with alternative notions of low-dimensional structure, e.g., manifold models \cite{XieHuaWil:12}, which have been considered previously in the context of high-dimensional change-point estimation (and more generally in signal processing), the setup described here has the virtue that one can employ efficient algorithms for convex optimization methods and one can appeal to insights from convex geometry in developing and analyzing algorithms for high-dimensional change-point estimation.  We discuss the relevant concepts in the next two subsections.


\subsection{Minkowski functional and proximal operators} \label{sec:minkowskifunc}

A key feature of our change-point estimation algorithm is the incorporation of a signal denoising step that exploits knowledge of the atomic set $\mathcal{A}$. To formally define the denoising operation, we consider the Minkowski functional $\|\cdot \|_{\mathcal{C}} : \mathbb{R}^p \mapsto [0,\infty]$
\begin{equation} \label{eq:minkowski}
\| \bfx \|_{\mathcal{C}} :=\inf \{ t: \bfx \in t \mathcal{C}, ~ t > 0 \},
\end{equation}
defined with respect to a convex set $\mathcal{C} \subset \R^p$ such that $\mathcal{A} \subseteq \mathcal{C}$, as discussed in \cite{CRPW:12}.  As $\mathcal{C}$ is convex, the Minkowski functional $\| \cdot \|_{\mathcal{C}}$ is also convex.  This function is also called the gauge function in the convex analysis literature \cite{Roc:70}.  For a given $\bfy \in \mathbb{R}^p$ and a convex set $\mathcal{C}$, we consider denoisers specified in terms of the following \emph{proximal operator}:
\begin{equation} \label{eq:proxoperator}
\hat{\bfx} =  \argmin_{\bfx \in \mathbb{R}^p} \frac{1}{2}  \| \bfy - \bfx \|_2^{2} + \lambda \| \bfx \|_{\mathcal{C}}.
\end{equation}
As $\| \cdot \|_{\mathcal{C}}$ is a convex function, this optimization problem is a convex program.  To obtain a proximal operator that enforces signal structure in the denoising operation, the set $\mathcal{C}$ is usually taken to be the tightest convex set containing the atomic set $\mathcal{A}$, i.e., $\mathcal{C}= \mathrm{conv}(\mathcal{A})$.  With $\mathcal{C}= \mathrm{conv}(\mathcal{A})$, the resulting Minkowski functional is called the \emph{atomic norm}\footnote{For \eqref{eq:minkowski} to formally define a norm, we would also need the set $\mathcal{A}$ to be centrally symmetric. Nevertheless, the results in the remainder of the paper hold without this condition, so we use ``norm'' with an abuse of terminology.} with respect to $\mathcal{A}$, and the associated proximal operator \eqref{eq:proxoperator} is called \emph{atomic norm thresholding} \cite{BhaTanRec:13}.  The atomic norm has been studied in the approximation theory literature for characterizing approximation rates associated with best $k$-term approximants \cite{Bar:93,DevTem:96,Jon:92,Pis:81}, and its convex-geometric properties were investigated in \cite{CRPW:12} in the context of ill-posed linear inverse problems.  When $\mathcal{A} = \{\pm \textbf{e}_i | 1\leq i \leq p\}$ is the collection of signed standard basis vectors, the atomic norm with respect to $\mathcal{A}$ is simply the $\ell_1$-norm in $\mathbb{R}^p$.  Similarly, the atomic norm corresponding to unit-Euclidean-norm rank-one matrices is the matrix nuclear norm.  More generally, one can define atomic norms associated to other types of structured objects such as permutation matrices, low-rank tensors, orthogonal matrices, and signed vectors; see \cite{CRPW:12} for a detailed list.  Atomic norm thresholding naturally generalizes soft-thresholding based on the $\ell_1$-norm for sparse signals to a more general denoising operation for the types of structured signals described here.

One exception to the rule of thumb of choosing $\mathcal{C} = \mathrm{conv}(\mathcal{A})$ arises if the atomic norm is intractable to represent, e.g., the tensor nuclear norm \cite{HilLim:13}.  That is, although these norms are convex functions, computing them may in general be computationally intractable.  To overcome such difficulties, a natural approach described in \cite{ChaJor:13,CRPW:12} is to consider Minkowski functionals of convex sets $\mathcal{C}$ that contain $\mathcal{A}$ and that are efficient to represent, i.e., further tractable convex relaxations of $\mathrm{conv}(\mathcal{A})$.


Finally, to avoid dealing with technicalities in degenerate cases, we assume throughout the remainder of the paper that the set $\text{conv}(\mathcal{A}) \subset \mathbb{R}^p$ is a solid convex set containing the origin in its interior.  Consequently, we have that $\| \bfx\|_{\mathcal{C}} < \infty$ for all $\bfx \in \mathbb{R}^p$.

\subsection{Summary parameters in signal denoising}

Next we describe the relevant convex-geometric concepts for analyzing the performance of proximal denoising operators.  For $\bfx \in \mathbb{R}^p$, the \emph{Gaussian distance} $\eta_{\mathcal{C}}(\bfx)$ \cite{FoyMac:13,OymHas:12} with respect to a norm $\| \cdot\|_{\mathcal{C}}$ is defined as
\begin{equation}\label{eq:gaussiandistance}
\eta_{\mathcal{C}}(\bfx) := \inf_{\lambda \geq 0} \biggl\{ \underset{\textbf{g} \sim \mathcal{N}(0, I_{p\times p})}{\expc} [ \mathrm{dist}(\textbf{g},\lambda \cdot \partial \| \bfx \|_{\mathcal{C}}) ] \biggr\}.
\end{equation}
Here $\mathrm{dist}(\textbf{g},\partial\| \bfx \|_{\mathcal{C}}) := \inf_{\textbf{w} \in \partial\| \bfx \|_{\mathcal{C}}} \| \textbf{w} - \textbf{g} \|_2$ denotes the distance of $\textbf{g}$ from the set $\partial\| \bfx \|_{\mathcal{C}}$, where $\partial\| \bfx \|_{\mathcal{C}}$ is the subdifferential of the function $\|\cdot \|_{\mathcal{C}}$ at the point $\bfx$ \cite{Roc:70}.  We relate the Gaussian distance to the Gaussian width \cite{Gor:88} in Appendix \ref{sec:gaussiandistwidth} by extending a result in \cite{FoyMac:13}.

The Gaussian distance $\eta_{\mathcal{C}}(\bfx)$ is useful for characterizing the performance of the proximal denoising operator \eqref{eq:proxoperator} \cite{BhaTanRec:13,ChaJor:13,OymHas:12}.  Specifically, suppose $\hat{\bfx} =  \argmin_{\bfx \in \mathbb{R}^p} \frac{1}{2}  \| \bfx^\star + \bfe - \bfx \|_2^{2} + \lambda \| \bfx \|_{\mathcal{C}}$, then the error between $\hat{\bfx}$ and $\bfx^\star$ is bounded as \cite{OymHas:12}:
\begin{equation*}
 \| \bfx^{\star}-\hat{\bfx} \|_2 \leq  \mathrm{dist}(\bfe,\lambda \cdot \partial \|\bfx^{\star}\|_{\mathcal{C}} ).
\end{equation*}
Taking expectations with respect to $\bfe$ and subsequently optimizing the resulting bound with respect to $\lambda$ yields the Gaussian distance \eqref{eq:gaussiandistance}.  We prove a generalization of this result in Appendix \ref{sec:prep}, which is relevant to the analysis of the change-point estimation algorithm proposed in Section \ref{sec:algorithm}.

As we discuss in Section \ref{sec:mainresults}, the combination of the proximal denoising operator with a suitable filtering step leads to a change-point estimation procedure with performance guarantees in terms of $\eta_{\mathcal{C}}(\bfx)$ rather than $\sqrt{p}$.  This point is significant because for many examples of structured signals that are encountered in practice, it is typically the case that $\eta_{\mathcal{C}}(\bfx) \ll \sqrt{p}$. For example, if $\bfx$ is an $s$-sparse vector in $\mathbb{R}^p$ then proximal denoising via the $\ell_1$-norm gives $\eta_{\ell_1} (\bfx) \leq \sqrt{2s \log (p/s) + 3s/2}+7$ \cite{CRPW:12,FoyMac:13,RudVer:06,Sto:09}.  Similarly, if $\bfx$ is a rank-$r$ matrix in $\mathbb{R}^{d \times d}$ then proximal denoising via the matrix nuclear norm gives $\eta_{\mathrm{nuc}} (\bfx)\leq \sqrt{6rd}+7$ \cite{CRPW:12,FoyMac:13,OymHas:11,RecXuHas:11}.

In order to state performance guarantees for a \emph{sequence} of observations, we extend the definition of $\eta_{\mathcal{C}}$ to collections of vectors $\mathcal{X}=\{ \bfx^{\star}[1],\ldots,\bfx^{\star}[n] \}, \bfx^{\star}(i) \in \mathbb{R}^p$ as follows:
\begin{equation} \label{eq:etadefn}
\eta_{\mathcal{C}} (\mathcal{X}) := \inf_{\lambda \geq 0} \underset{\bfx^{\star}[t] \in \mathcal{X}}{\max} \biggl\{ \underset{\textbf{g} \sim \mathcal{N}(0, I_{p\times p})}{\expc} [ \mathrm{dist}(\textbf{g},\lambda \cdot \partial \| \bfx^{\star} [t] \|_{\mathcal{C}}) ] \biggr\}.
\end{equation}

%


\section{Convex Programming for Change-Point Estimation} \label{sec:mainresults}

In this section, we describe our algorithm for high-dimensional change-point estimation by combining the filtered derivative method with proximal denoising.  We state the main theorem that characterizes the accuracy of the estimated set of change-points, and we outline the proof, with the full details given in the Appendix.


\subsection{Motivation}
We begin by highlighting some of the difficulties that arise in change-point estimation as a result of the high-dimensionality of the observations.  In order to frame our discussion concretely, we consider the prominent and widely-employed class of change-point estimation techniques based on the filtered derivative algorithm \cite{AntHus:94,BenBas:84,Ber:00,BerFhiGui:11}, although similar difficulties arise with other approaches as well.  The \emph{filtered derivative} method detects changes based on an application of a pairwise difference operator to the output of a suitable low-pass filter applied to the sequence of observations.  For simplicity, we describe a particular low-pass filter that is given by the sample mean of the observations over a small window (again, elaborations on this scheme are possible, with qualitatively similar conclusions).  Formally, consider the following sequence defined at time $t$ by computing differences of sample means over windows of size $\theta$:
\begin{equation} \label{eq:filteredderivative}
\mathrm{FD}_{\theta}[t] = -\frac{1}{\theta} \sum^{t}_{i=t-\theta+1} Y[i] + \frac{1}{\theta} \sum^{t+\theta}_{i=t+1} Y[i].
\end{equation}
Locations at which $\mathrm{FD}_{\theta}[t]$ has large magnitude (i.e., above a suitably chosen threshold) are declared as change-points.

This approach is well-suited for settings with sequences of scalar-valued signals, i.e., each $\bfx^{\star}[t]$ is scalar; see \cite{AntHus:94,Ber:00,BerFhiGui:11} for detailed analyses.  However, if applied directly to the high-dimensional setting, the underlying sequence of signals $\bfx^{\star}[t] \in \R^p$ is required to remain stationary over time scales on the order of $p$ so that changes can be reliably estimated.  This requirement is unfortunately not realistic for practical purposes, e.g., in image processing applications one typically encounters $p \approx 10^6$.  As such, it is desirable to develop an algorithm that detects changes in sequences of high-dimensional observations reliably even if the signal does not remain stationary over long time scales.



\subsection{Our approach to high-dimensional change-point estimation} \label{sec:algorithm}




We base our method on the principle that more effective signal denoising by exploiting the low-dimensional structure underlying the sequence $\bfx^{\star}[t]$ enables improved change-point estimation.  The formal steps of our algorithm for obtaining an estimate $\hat{\tau}$ of $\tau^{\star}$ are as follows:
\begin{enumerate}
\item \textbf{[Input]:} $\{\bfy[t] \}_{t=1}^n$ the sequence of signal observations, a choice of parameters $\theta,\gamma,\lambda$ to be employed in the algorithm, and a specification of a convex set $\mathcal{C}$.
\item \textbf{[Filtering]:} Compute the sample means $\bar{\bfy}[i] = \frac{1}{\theta} \sum^{i+\theta -1}_{t=i} \bfy[t], 1\leq i \leq n -\theta +1$.
\item \textbf{[Denoising]:} Let $\hat{\bfx}[t], 1\leq t \leq n - \theta +1$ be the solutions to the following convex optimization problems:
\begin{equation} \label{eq:anst}
\hat{\bfx}[t] =  \argmin_{\bfx \in \mathbb{R}^p} \frac{1}{2}  \| \bar{\bfy}[t] - \bfx \|_2^{2} + \lambda \| \bfx \|_{\mathcal{C}}.
\end{equation}
\item \textbf{[Differencing]:} Compute $S[t] := \| \hat{\bfx}[t+1] - \hat{\bfx}[t-\theta+1] \|_2$ for $\theta \leq t \leq n - \theta$.
\item \textbf{[Thresholding]:} For all $t$ such that $S[t] < \gamma$, set $S[t] = 0$.
\item \textbf{[Output]:} $\hat{\tau}=\{ \hat{t}_i\}$. Group time intervals such that consecutive non-zero entries in $S[t]$ are at most $\theta$ apart, and in each group select the index $\hat{t}$ corresponding to the largest value of $S$ as an estimate of the location of a change.
\end{enumerate}
Observe that the proximal denoising step is applied before the differencing step.  This particular integration of proximal denoising and the filtered derivative ensures that the differencing operator is applied to estimates $\bar{\bfx}[t]$ that are closer to the underlying signal $\bfx^\star[t]$ than the raw averages $\bar{\bfy}[t]$ (due to the favorable denoising properties of the proximal denoiser). As discussed in Theorem \ref{thm:changepoint}, this leads to improved change-point performance in comparison to a pure filtered derivative method.  However, the analysis of our approach is complicated by the introduction of the proximal denoising step; we discuss this point in greater detail in Section \ref{sec:quickproof}.

%

The parameter $\theta$ determines the window over which we compute the sample mean, and it controls the resolution to which we estimate change-points.  A larger value of $\theta$ allows the algorithm to detect small changes, although if $\theta$ is chosen too large, multiple change-points may be mistaken for a single change-point.  A smaller choice of $\theta$ increases the resolution of the change-point estimates, but small changes cannot be reliably detected.  The parameter $\gamma$ specifies the threshold for declaring changes, and it governs the size of the change-points that can be reliably estimated.  A small choice of $\gamma$ allows the algorithm to detect smaller changes but it also increases the occurrence of false positives.  Conversely, a larger value of $\gamma$ reduces the number of false positives, but only those changes that are sufficiently large in magnitude may be detected by the algorithm (i.e., the number of false negatives may increase).  In Theorem \ref{thm:changepoint}, we give precise guidelines for the choices of the parameters $(\theta,\gamma,\lambda)$ to guarantee reliable change-point estimation under suitable conditions via the method described above.




\begin{theorem} \label{thm:changepoint} Consider a sequence of observations $\bfy[t] = \bfx^{\star}[t]+\bfe[t]$, $t =1,\ldots,n$, where each $\bfx^{\star}[t] \in \mathbb{R}^p$ and each $\bfe[t] \sim \mathcal{N}(0,\sigma^2 I_{p\times p})$ independently. Let $\tau^{\star} \subset \{1,\dots,n\}$ be such that $t \in \tau^{\star} \Leftrightarrow \bfx^{\star}[t] \neq \bfx^{\star}[t+1]$, let $\Delta_{\min} = \min_{t \in \tau^{\star}} \| \bfx^{\star}[t] - \bfx^{\star}[t+1] \|_2 $, let $T_{\min} = \min_{t_i, t_j \in \tau^{\star},t_i \neq t_j} | t_i -t_j|$, and let $\mathcal{X}=\{ \bfx^{\star}[1],\ldots,\bfx^{\star}[n] \}$. Suppose $\Delta_{\min}$ and $T_{\min}$ satisfy
\begin{equation} \label{eq:changepoint_sufficient}
\Delta_{\min}^{2} T_{\min} \geq 64 \sigma^2 \{ \eta_{\mathcal{C}}(\mathcal{X}) + r\sqrt{ 2 \log n} \}^2
\end{equation}
for some $r>1$ and some convex set $\mathcal{C}$, where $\eta_{\mathcal{C}}(\mathcal{X})$ is as defined in \eqref{eq:etadefn}, and $\tau^* \subset \{ T_{\min}/4,\ldots,n-T_{\min}/4 \}$. Suppose we apply our change-point estimation algorithm with any choice of parameters $\theta,\gamma$, and $\lambda$ satisfying
\begin{enumerate}
\item $T_{\min}/4 \geq \theta$,
\item $\Delta_{\min}/2 \geq \gamma \geq 2\frac{\sigma}{\sqrt{\theta}}\{  \eta_{\mathcal{C}}(\mathcal{X}) + r\sqrt{2\log n}\}$, and
\item $ \lambda = \frac{\sigma}{\sqrt{\theta}}\underset{\tilde{\lambda}}{\argmin}  \underset{\bfx^{\star}[t] \in \mathcal{X}}{\max} \biggl\{ \underset{g \sim \mathcal{N}(0,I_{p \times p})}{\expc} [ \mathrm{dist}(g, \tilde{\lambda} \cdot \partial \| \bfx^{\star}[t] \|_{\mathcal{C}}) ] \biggr\}$.
\end{enumerate} Then the algorithm recovers an estimate of the change-points $\hat{\tau}$ satisfying
\begin{enumerate}
\item $|\hat{\tau}| = |\tau^{\star}|$
\item $| \hat{t}_i-t_i^{\star} | \leq \min \{ (4r\sqrt{\log n} / \eta_{\mathcal{C}}(\mathcal{X})+4) \frac{\sigma \eta_{\mathcal{C}}(\mathcal{X})}{\Delta{\min}}  \sqrt{\theta},\theta \}$ for all $i$, where $\hat{t}_i$ and $t_i^{\star}$ are the $i$-th elements of $\hat{\tau}$ and $\tau^{\star}$ when ordered sequentially,
\end{enumerate}
with probability greater than $1- 5n^{1-r^2}$.
\end{theorem}

\paragraph{Remark.} If condition \eqref{eq:changepoint_sufficient} is satisfied then the choice of $\theta = T_{\min}/4$ and $\gamma = \Delta_{\min}/2$ satisfies the requirements in Theorem \ref{thm:changepoint}. 

As a concrete illustration of this theorem, if each element $\bfx^\star[t] \in \R^p, ~ t=1,\dots,n$ of the signal sequence is a vector consisting of at most $s$ nonzero entries, then our algorithm (with a proximal denoiser based on the $\ell_1$-norm) estimates change-points reliably under the condition $\Delta^2_{\min} T_{\min} \gtrsim \sigma^2 (s \log(\tfrac{p}{s}) + \log(n))$.  Similarly, if each element $\bfx^\star[t] \in \R^{d \times d}, ~ t=1,\dots,n$ is a matrix with rank at most $r$, then our algorithm (with a proximal denoiser now based on the nuclear norm) provides reliable change-point estimation performance under the condition $\Delta^2_{\min} T_{\min} \gtrsim \sigma^2 (rd + \log(n))$.

\subsection{Proof of Theorem \ref{thm:changepoint}} \label{sec:quickproof}

The proof broadly proceeds by bounding the probabilities of the following three events:
\begin{align} \label{eq:eventse1}
E_1: &= \{ S[t] \geq \gamma, \forall t \in \tau^{\star}\} \\
E_2 :&=\{ S[t] < \gamma, \forall t \in \tau_{\text{far}} \} \label{eq:eventse2} \\
E_3 : &= \{\| \hat{\bfx}[t+1] - \hat{\bfx}[t -\theta +1] \|_2 > \| \hat{\bfx}[t+1 +\delta] - \hat{\bfx}[t -\theta +1 +\delta] \|_2 , \forall (t,\delta)  \in \tau_{\text{buffer}} \}.\label{eq:eventse3}
\end{align}
Here $\tau_{\text{far}} = \{i: \theta\leq i \leq n-\theta,  |i-j|> \theta, j \in \tau^{\star} \}$ and $\tau_{\text{buffer}} = \{ (t_i^{\star}, \delta ): t_i^{\star} \in \tau^{\star},  \theta \geq |\delta| >  (4r\sqrt{\log n} / \eta_{\mathcal{C}}(\mathcal{X})+4) \frac{\sigma \eta_{\mathcal{C}} (\mathcal{X})}{\Delta_{\min}}\sqrt{\theta} \}$. Note that $\tau_{\text{buffer}}$ defines a non-empty set if $\theta >  (4r\sqrt{\log n} / \eta_{\mathcal{C}}(\mathcal{X})+4)^2 \sigma^2 \eta^2_{\mathcal{C}} (\mathcal{X})/\Delta_{\min}^2$. The event $E_1$ corresponds to the atomic-norm-thresholded derivative exceeding the threshold $\gamma$ for all change-points, while event $E_2$ corresponds to the atomic-norm-thresholded derivative \emph{not} exceeding the threshold $\gamma$ in regions ``far away" from the change-points.  Bounding the probabilities of these two events is sufficient for a weaker recovery guarantee than is provided by Theorem \ref{thm:changepoint}, which is that any estimated change-point $\hat{t}\in \hat{\tau}$ will be within $\theta$ of an actual change-point $t^{\star} \in \tau^{\star}$.  However, the selection of the \emph{maximum} derivative in Step 6 of the algorithm often leads to far more accurate estimates of the locations of change-points.  To prove that this is the case, we consider the event $E_3$ corresponding to the atomic-norm-thresholded derivative at the change-point being larger than the atomic-norm-thresholded derivatives at other points that are still within a window of $\theta$ but outside a small buffer region around the change-point.



The next proposition gives bounds on the probabilities of the events $E_1,E_2,E_3$:

\begin{proposition} \label{thm:probboundsE1E2E3}
Under the setup and conditions of Theorem \ref{thm:changepoint}, we have the following bounds:
\begin{align}
\prob (E_{1}^c ) \leq 2n^{1-r^2}, ~~~ \prob (E_{2}^c ) \leq 2n^{1-r^2}, ~~~ \prob (E_{3}^c ) \leq n^{1-r^2}.
\end{align}
The events $E_1,E_2,E_3$ are defined in \eqref{eq:eventse1}, \eqref{eq:eventse2}, and \eqref{eq:eventse3}.
\end{proposition}

The proof of Proposition \ref{thm:probboundsE1E2E3} is given in the Appendix, and it involves overcoming two difficulties.  First, if the filtering operator is applied over a window containing a change-point in Step 2, the average $\bar{\bfy}[t]$ is in effect a superposition of two structured signals corrupted by noise.  This necessitates the analysis of the performance of a proximal denoiser applied to a noisy superposition of structured signals rather than to a single structured signal corrupted by noise.  The second (more challenging) complication arises due to the fact that the differencing operator is applied to the result of a nonlinear mapping of the observations (via the proximal denoiser) rather than to just a linear average of the observations as in a standard filtered derivative framework.  We address these difficulties by exploiting certain properties of the proximal operator such as its non-expansiveness and its robustness to perturbations. Assuming Proposition \ref{thm:probboundsE1E2E3}, the proof of Theorem \ref{thm:changepoint} proceeds as follows:


\emph{Proof of Theorem \ref{thm:changepoint}} From Proposition \ref{thm:probboundsE1E2E3} and the union bound we have that $\prob(E_1 \cap E_2 \cap E_3) \geq 1-5 n^{r^2 -1}$.  We condition on the event $E_1 \cap E_2 \cap E_3$ to complete the proof.  Specifically, conditioning on $E_1$ ensures that we have $S[t] \geq\gamma$ for all $t \in \tau^{\star}$ after Step 5.  Conditioning on $E_2$ implies that all entries of $S$ outside a window of $\theta$ from any change-point are set to $0$ after Step 5.  Hence, after Step 6 all non-zero entries of $S$ that are within a window of at most $\theta$ around a change-point will have been grouped together (for all change-points), which implies that $|\hat{\tau} |=|\tau^{\star}|$.  Finally, conditioning on the event $E_3$ implies that $S[t]$ is larger than $S[t+\delta]$ for all $\delta$ such that $ (4r\sqrt{\log n} / \eta_{\mathcal{C}}(\mathcal{X})+4)  \frac{\sigma \eta_{\mathcal{C}}(\mathcal{X})}{\Delta{\min}}\sqrt{\theta} < |\delta| \leq \theta$ and for all $t\in \tau^{\star}$. Thus, our estimate of the change-point at $t \in \tau^{\star}$ is at most $\min\left\{(4r\sqrt{\log n} / \eta_{\mathcal{C}}(\mathcal{X})+4)  \frac{\sigma \eta_{\mathcal{C}}(\mathcal{X})}{\Delta{\min}}\sqrt{\theta},\theta\right\}$ away, which concludes the proof. $\qed$


\subsection{Signal reconstruction}

Based on the estimated set of change-points $\hat{\tau}$, it is straightforward to obtain good reconstructions of the signal away from the change-points via proximal denoising \eqref{eq:proxoperator}.  This result follows from the analysis in \cite{BhaTanRec:13,OymHas:12}, but we state and prove it here for completeness:


\begin{proposition} \label{thm:squarederror}
Suppose that the assumptions of Theorem \ref{thm:changepoint} hold. Let $t_1, t_2 \in \hat{\tau}$ be two consecutive estimates of change-points, let $ \lambda' = \frac{\sigma}{\sqrt{t_2 - t_1 - 2\theta}} \argmin_{\tilde{\lambda}}  \max_{\bfx^{\star}[t] \in \mathcal{X}} \{ \expc_{g \sim \mathcal{N}(0,I_{p \times p})} [ \mathrm{dist}(g, \tilde{\lambda} \cdot \partial \| \bfx^{\star}[t] \|_{\mathcal{C}}) ] \}$, and let $\bar{\bfy} =\frac{1}{t_2 - t_1- 2\theta} \sum^{t_2 - \theta}_{t=t_1 + \theta +1} \bfy[t] $.  Denote the solution of the proximal denoiser \eqref{eq:proxoperator} applied to $\bar{\bfy}$ as $\bar{\bfx}$:
\begin{equation*}
\bar{\bfx} := \argmin_{\bfx} \frac{1}{2} \bigl \| \bar{\bfy} -\bfx \bigr\|_2^2 + \lambda' \| \bfx \|_{\mathcal{C}}.
\end{equation*}
Letting $\bar{\bfx}$ be the estimate of the signal ${\bfx}^{\star}[t]$ over the interval $\{t_1 +\theta + 1,\dots,t_2 - \theta\}$, we have that
\begin{equation*}
\| \bar{\bfx} - {\bfx}^{\star}[t]  \|^2_2 \leq \frac{2\sigma^2}{t_2 - t_1 - 2\theta} \{ \eta_{\mathcal{C}}(\mathcal{X})^2 + s^2 \}
\end{equation*}
with probability greater than $1- 4 n^{1-r^2} - \exp(-s^2 /2)$, for all $t_1 + \theta +1 \leq t \leq t_2 - \theta$.
\end{proposition}

The proof of Proposition \ref{thm:squarederror} is given in the Appendix.  The quantities $\Delta_{\min}$ and $T_{\min}$ determine the accuracy of the change-point estimates from Theorem \ref{thm:changepoint}.  Proposition \ref{thm:squarederror} demonstrates that, in addition to these quantities, the \emph{duration of stationarity} of the signal also determines the accuracy of signal reconstruction away from the change-points.


\section{Tradeoffs in High-Dimensional Change-Point Estimation} \label{sec:tradeoffs}

Data analysis in practice involves a range of challenges beyond the high-dimensionality of the observations that motivated our development in this paper.  For example, in change-point estimation in financial time series, one is typically faced with additional difficulties arising from the extremely rapid rate at which the data are acquired and the requirement that the data be processed in an ``online'' fashion, i.e., the change-point estimation procedure must process the incoming data in ``real time.''  In some settings rapid variations in an underlying phenomenon trigger frequent changes in the sequence of observations, while in other cases small changes in a signal can be difficult to detect when severely corrupted by noise.  In this section, we describe adaptations of the algorithm proposed in Section \ref{sec:mainresults} to handle some of these challenges.  Specifically, Theorem \ref{thm:changepoint} suggests a number of performance tradeoffs that can be obtained in change-point estimation problems by employing suitable variations of our algorithm.  We demonstrate the utility of these modifications in addressing some of the difficulties mentioned above, which highlights the versatility of our approach.


\subsection{Change-point frequency and size tradeoffs} \label{sec:freqsizetradeoffs}

The appearance of the term $\Delta_{\min}^2 T_{\min}$ in \eqref{eq:changepoint_sufficient} suggests an explicit relation between the minimum time span between changes, the minimum size of a change, and estimation accuracy. To illustrate the tradeoffs between $\Delta_{\min}$ and $T_{\min}$ clearly, we fix the complexity parameter $\eta_{\mathcal{C}}(\mathcal{X})$ and the number of observations $n$ in this discussion.  As a consequence, the quantity $\Delta_{\min}^{2} T_{\min}$ in \eqref{eq:changepoint_sufficient} can be interpreted as a \emph{resolution} on the types of changes that can be detected.  Specifically, Theorem \ref{thm:changepoint} guarantees reliable estimation of changes whenever $\Delta_{\min}^{2} T_{\min}$ is sufficiently large even if one of $\Delta_{\min}$ or $T_{\min}$ may be small. There are two separate regimes where this observation has interesting implications.

In settings in which $\Delta_{\min}$ is small, i.e., there are change-points where the size of the change is small, Theorem \ref{thm:changepoint} guarantees that all the changes can be detected reliably so long as the distance between change-points ($T_{\min}$) is sufficiently large.  In order to accomplish this, one is required to choose a sufficiently small threshold parameter $\gamma$ (for Step 5 of the procedure) and a suitably large averaging window $\theta$ (for Step 2) with $\theta \lesssim T_{\min}$, in accordance with the requirements of Theorem \ref{thm:changepoint}.  By smoothing over large windows of size $\theta$ and subsequently applying a proximal denoiser, even small-sized changes can be detected as long as the averaging window does not include multiple change-points (hence the condition that $\theta \lesssim T_{\min}$).  The downside with choosing a large value for the parameter $\theta$ is that we do not resolve the locations of the change-points well; in particular, we estimate the locations of each of the change-points to within a resolution of about $\sqrt{\theta}$.  However, detecting small-sized changes in a sequence corrupted by noise necessitates the computation of averages over large windows in Step 2 of our algorithm in order to distinguish genuine changes from spurious ones.  Therefore, the low resolution to which we estimate the locations of change-points is the price to pay for estimating the number of change-points exactly in settings in which some of the changes may be small in size.

In a similar manner, if changes occur frequently in a signal sequence, i.e., the distance between change-points $T_{\min}$ is small, Theorem \ref{thm:changepoint} guarantees that all the changes can be detected reliably if the size of each change $\Delta_{\min}$ is sufficiently large.  In such cases, the averaging window parameter $\theta$ must be chosen to be sufficiently small while the threshold parameter $\gamma$ must be appropriately large with $\gamma \lesssim \Delta_{\min}$, as prescribed in Theorem \ref{thm:changepoint}.  The choice of a small value for $\theta$ ensures that we do not smooth the observation sequence over windows that contain multiple change-points in Step 2 of our method.  However, this restriction of the averaging window size implies that the proximal denoiser in Step 3 is applied to the average of a small number of observations, which negatively impacts its performance.  This limitation underlies the choice of a large value for the threshold parameter $\gamma$ in Step 5 of the algorithm, which ensures that spurious changes resulting from denoising over small windows do not impact the performance of our algorithm.  Consequently, the size of each change must be sufficiently large (as required by the condition that $\gamma \lesssim \Delta_{\min}$) so that the change-points can be reliably estimated from a few noisy observations.  Thus, the size of each change must be sufficiently large in settings with frequent changes so that the number of change-points can be reliably estimated.

%


\subsection{Computational complexity and sample size tradeoffs}\label{sec:compstattradeoffs}

In change-point estimation tasks arising in many contemporary problem domains (e.g., financial time series), one is faced with a twin set of challenges: $(a)$ the number of observations $n$ may be quite large due to the increasingly higher frequencies at which data are acquired (this is in addition to the high dimensionality of each observation), and $(b)$ the requirement that these large datasets be processed online or in real time. Consequently, as the number of observations per unit time grows, it is crucial that we adopt a \emph{simpler} algorithmic strategy -- i.e., a method requiring a smaller number of computational steps per observation -- so that the overall computational complexity of our algorithm does not grow with the number of observations.  In this section we describe a convex relaxation approach to adapt the algorithm described in Section \ref{sec:mainresults} to achieve a tradeoff between the number of observations and the overall computational complexity of our procedure; in particular, we demonstrate that in certain change-point estimation problems one can even achieve an overall \emph{reduction} in computational runtime as the number of observations grows.  Our method is based on the ideas presented in \cite{ChaJor:13,CRPW:12} in the context of statistical denoising and of linear inverse problems; here we demonstrate the utility of those insights in change-point estimation.  We note that other researchers have also explored the idea of trading off computational resources and sample size in various inferential problems such as binary classifier learning \cite{BirSha:12,DecGolRon:98,Ser:00,ShaShaTro:12}, in sparse principal component analysis \cite{AmiWai:09,BerRig:13,KBRS:11}, in model selection \cite{AgaBarDuc:12}, and in linear regression \cite{SheLaf:13}.

\paragraph{A modified change-point estimation algorithm} For ease of analysis and exposition, we consider a modification of our change-point estimation procedure from Section \ref{sec:mainresults}.  Specifically, Step 6 of our algorithm is simplified so it only groups time indices corresponding to the nonzero entries of the thresholded derivative values, with consecutive time indices in a group at most $\theta$ apart (i.e., without further choosing the maximum element from each group).  Thus, our algorithm only produces windows that localize change-points instead of returning precise estimates of change-points.  The reason for restricting our attention to such a simplification is that the additional operation of choosing the maximum element in Step 6 of the original algorithm leads to unnecessary complications that are not essential to the point of the discussion in this section.  The performance analysis of this modified algorithm follows from Theorem \ref{thm:changepoint}, and we record this result next:

\begin{corollary} \label{thm:simplechagepoint}
Under the same setup and conditions as in Theorem \ref{thm:changepoint}, suppose that we perform the modified change-point estimation algorithm -- that is, Step 6 is simplified to only return groups of times indices, where consecutive time indices in a group are at most $\theta$ apart.  Then we have with probability greater than $1 - 4n^{1-r^2}$ that $(i)$ there are exactly $\tau^\star$ groups, and $(ii)$ the $j$'th group $g_j \subset \{\theta, \dots, n-\theta\}$ is such that $|t^\star_j - \tilde{t}| \leq \theta$ for all $\tilde{t} \in g_j$.
\end{corollary}


In order to concretely illustrate tradeoffs between the number of observations and the overall computational runtime, we focus on the following stylized change-point estimation problem.  Consider a continuous-time piecewise constant signal $\bfx^{\star}(T) \in \R^p, ~ T\in(0,1]$ defined as:
\begin{equation*}
\bfx^{\star}(T) = {\bfx^{\star}}^{(i)}, ~~~ T \in (T_i, T_{i+1}].
\end{equation*}
That is, the signal $\bfx^{\star}(T)$ takes on the value ${\bfx^{\star}}^{(i)} \in \R^p$ identically for the entire time interval $T \in (T_i,T_{i+1}]$ for $i=1,\dots,k$.  Here $i=1,\dots,k$ and the time indices $\{T_i\}_{i=1}^{k+1}$ are such that $0=T_1 \leq \dots \leq T_{k+1} = 1$.  Suppose we have two collections of noisy observations obtained by sampling the signal $\bfx^{\star}(T)$ at equally-spaced points $\frac{1}{n}$ apart and $\frac{1}{kn}$ apart for some positive integers $k > 1$ and $n$ (we assume that $n \gg \frac{1}{T_{i+1}-T_i}$ for all $i$):
\begin{eqnarray*}
\bfy^{(1)}[t] &=& \bfx^{\star}\left(\frac{t}{n}\right) + \epsilon[t], ~~~ t=1,\dots,n \\ \bfy^{(2)}[t] &=& \bfx^{\star}\left(\frac{t}{kn}\right) + \tilde{\epsilon}[t], ~~~ t=1,\dots,kn.
\end{eqnarray*}
Here $\epsilon[t], \tilde{\epsilon}[t] \sim \mathcal{N}(0,\sigma^2 I_{p \times p})$ are i.i.d. Gaussian noise vectors.  In words, $\bfy^{(2)}[t]$ is a $k$-times more rapidly sampled version of $\bfx^{\star}(T)$ than $\bfy^{(1)}[t]$.  As a result, the sequence $\bfy^{(1)}[t]$ consists of $n$ observations and the sequence $\bfy^{(2)}[t]$ consists of $kn$ observations.  Consequently, it may seem that estimating the change-points in the sequence $\bfy^{(2)}[t]$ requires at least as many computational resources as the estimation of change-points in $\bfy^{(1)}[t]$.  However, when viewed from the prism of Corollary \ref{thm:simplechagepoint} and Theorem \ref{thm:changepoint}, the sequence $\bfy^{(2)}[t]$ is in some sense more favorable than $\bfy^{(1)}[t]$ for change-point estimation -- specifically, if the minimum distance between successive change-points underlying the sequence $\bfy^{(1)}[t]$ is $T_{\min}$, then the minimum distance between successive change-points in $\bfy^{(2)}[t]$ is $k T_{\min}$, i.e., larger by a factor of $k$. (Note that $\Delta_{\min}$ for both sequences remains the same.)

Let $\mathcal{X} =\{ \bfx^{\star(1)},\ldots,\bfx^{\star (k)} \}$. Applying the modified change-point estimation algorithm described above to the sequence $\bfy^{(1)}[t]$ with parameters $\theta_1 = T_{\min} / 4, \gamma_1 = \Delta_{\min}/ 2$ and with a proximal denoising step based on a convex set $\mathcal{C}$, we obtain reliable localizations of the change-points under the condition:
\begin{equation*}
\Delta_{\min}^{2} T_{\min} \geq 64 \sigma^2 \{ \eta_{\mathcal{C}}(\mathcal{X}) + r\sqrt{ 2 \log n} \}^2,
\end{equation*}
That is, we localize the change-points to a window of size $\theta_1 = T_{\min} / 4$.  Now suppose we apply the modified change-point algorithm to the sequence $\bfy^{(2)}[t]$ (note that this sequence is of length $kn$) with parameters $\theta_2 = kT_{\min} / 4, \gamma_2 = \Delta_{\min}/ 2$ and with a proximal denoising step based on the same convex set $\mathcal{C}$.  In this case, we reliably localize each change-point in $\bfy^{(2)}[t]$ to a window of size $\theta_2 = kT_{\min} / 4$ under the following condition:
\begin{equation} \label{eq:suffcondktimes}
\Delta_{\min}^{2} (k T_{\min}) \geq 64 \sigma^2 \{ \eta_{\mathcal{C}}(\mathcal{X}) + r\sqrt{ 2 \log n + 2 \log k} \}^2,
\end{equation}
The quality of the output in both cases is the same -- identifying changes in $\bfy^{(1)}[t]$ to a resolution of $\theta_1 = T_{\min} / 4$ is comparable to identifying changes in $\bfy^{(2)}[t]$ to a resolution of $\theta_2 = kT_{\min} / 4$, because $\bfy^{(2)}[t]$ is a $k$-times more rapidly sampled version of the continuous-time signal $\bfx^{\star}(T)$ in comparison to $\bfy^{(1)}[t]$.  On the computational side, our algorithm involves roughly $n$ applications of the proximal denoiser based on the set $\mathcal{C}$ in the case of $\bfy^{(1)}[t]$, and about $kn$ applications of the same proximal denoiser in the case of $\bfy^{(2)}[t]$.  Therefore, the overall runtime is higher in the case of $\bfy^{(2)}[t]$ than in the case of $\bfy^{(1)}[t]$.

Notice that the left-hand-side of the condition \eqref{eq:suffcondktimes} goes up by a factor of $k$.  We exploit this increased gap between the two sides of the inequality in \eqref{eq:suffcondktimes} to obtain a smaller overall computational runtime for estimating changes in the sequence $\bfy^{(2)}[t]$ than for estimating changes in the sequence $\bfy^{(1)}[t]$.  The key insight underlying our approach, borrowing from the ideas developed in \cite{ChaJor:13,CRPW:12}, is that we can employ a computationally cheaper proximal denoiser when applying our algorithm to the sequence $\bfy^{(2)}[t]$.  Specifically, for many interesting classes of structured signals, one can replace the proximal denoising operation with respect to the convex set $\mathcal{C}$ in step $(3)$ of our algorithm with a proximal operator corresponding to a \emph{relaxation} $\mathcal{B} \subset \R^p$ of the set $\mathcal{C}$, i.e., $\mathcal{B}$ is a convex set such that $\mathcal{C} \subset \mathcal{B}$.  For suitable relaxations $\mathcal{B}$ of the set $\mathcal{C}$, the proximal denoiser associated to $\mathcal{B}$ is more efficient to compute than the proximal denoiser with respect to $\mathcal{C}$, and further $\eta_{\mathcal{C}}(\mathcal{X}) < \eta_{\mathcal{B}}(\mathcal{X})$.  The reason for the second property is that, under appropriate conditions, the subdifferentials with respect to the gauge functions $\|\cdot\|_{\mathcal{B}}, \|\cdot\|_{\mathcal{C}}$ satisfy the condition that $\partial \| \bfx^{\star} \|_{\mathcal{B}} \subset \partial \| \bfx^{\star} \|_{\mathcal{C}}$ at signals of interest $\bfx^{\star} \in \mathcal{X}$.  We refer the reader to \cite{ChaJor:13} for further details, and more generally, to the convex optimization literature \cite{GouPT2010,Las2001,Par2000,Ren2006,SheA1990} for various constructions of families of tractable convex relaxations.


Going back to the sequence $\bfy^{(2)}[t]$, we can employ a proximal denoiser based on \emph{any} tractable convex relaxation $\mathcal{B}$ of the set $\mathcal{C}$ as long as the following condition (a modification of \eqref{eq:suffcondktimes}) for reliable change-point estimation is satisfied:
\begin{equation*}
\Delta_{\min}^2 (k T_{\min}) \geq 64 \sigma^2 \{ \eta_{\mathcal{B}}(\mathcal{X}) + r\sqrt{ 2 \log n + 2\log k} \}^2.
\end{equation*}
Indeed, if this condition is satisfied, we can still localize changes to a resolution of $kT_{\min} / 4$, i.e., the same quality of performance as before with a proximal denoiser with respect to the set $\mathcal{C}$. However, the computational upshot is that the number of operations required to estimate change-points in $\bfy^{(2)}[t]$ using the modified proximal denoising step is roughly $kn$ applications of the proximal denoiser based on the relaxations $\mathcal{B}$.  The contrast to $n$ applications of a proximal denoiser based on $\mathcal{C}$ for estimating change-points in the sequence $\bfy^{(1)}[t]$ can be significant if computing the proximal denoiser with respect to $\mathcal{B}$ is much more tractable than computing the denoiser with respect to $\mathcal{C}$.


We give an example in which such convex relaxations can lead to reduced computational runtime as the number of observations increases.  We refer the reader to \cite{ChaJor:13} for further illustrations in the context of statistical denoising, which can also be translated to provide interesting examples in a change-point estimation setting.  Specifically, suppose that the underlying signal set $\mathcal{X} = \{\boldsymbol a \boldsymbol a^{T} | \boldsymbol a \in \{-1,+1 \}^d \}$, i.e., the signal at each instant in time is a rank-one matrix formed as an outer product of signed vectors.  In this case, a natural candidate for a set $\mathcal{C}$ is the set of $d \times d$ \emph{correlation matrices}, which is also called the \emph{elliptope} in the convex optimization literature \cite{DezL1997}. One can show that each application of a proximal denoiser with respect to $\mathcal{C}$ requires $\mathcal{O}(d^{4.5})$ operations \cite{Hig:02}.  The $d \times d$ nuclear norm ball (scaled to contain all $d \times d$ matrices with nuclear norm at most $d$), which we denote as $\mathcal{B}$, is a relaxation of the set $\mathcal{C}$ of correlation matrices.  Interestingly, the distance $\eta_{\mathcal{B}}(\mathcal{X})$ is only a constant times larger (independent of the dimension $d$) than $\eta_{\mathcal{C}}(\mathcal{X})$ \cite{ChaJor:13}.  However, each application of a proximal denoiser with respect to $\mathcal{B}$ requires only $\mathcal{O}(d^{3})$ operations.  In summary, even if the increased sampling factor $k$ in our setup is larger than a constant (independent of the dimension $d$), one can obtain an overall reduction in computational runtime from about $\mathcal{O}(nd^{4.5})$ operations to about $\mathcal{O}(knd^{3})$ operations.

\section{Numerical Results} \label{sec:numericalresults}

We illustrate the performance of our change-point estimation algorithm with two numerical experiments on synthetic data.




\begin{figure}
\centering
\includegraphics[width=0.8\textwidth]{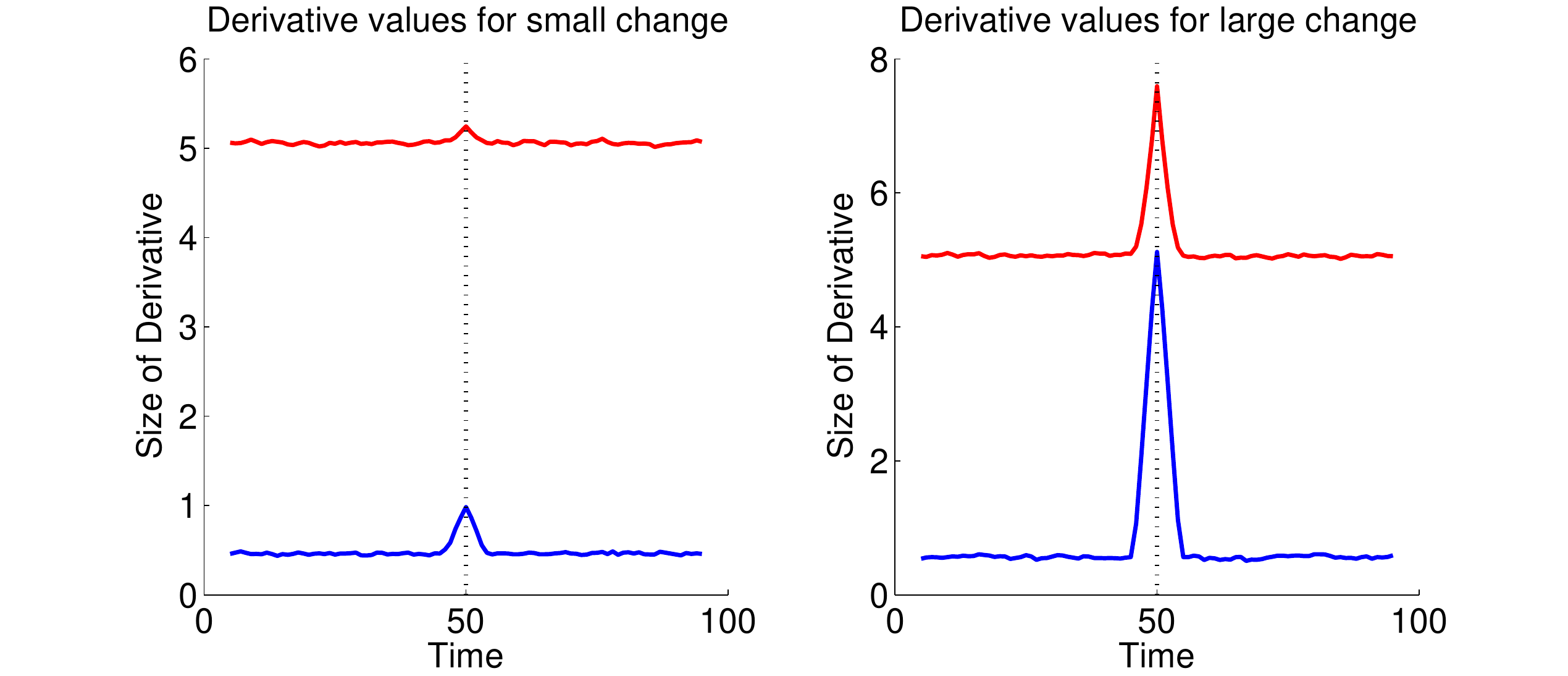}
\caption{Experiment contrasting our algorithm (in blue) with the filtered derivative approach (in red): the left sub-plot corresponds to a small-sized change and the right sub-plot corresponds to a large-sized change.}
\label{fig:PlantedLRcombine}
\end{figure}



\textbf{A contrast between our approach and the filtered derivative.}  The objective of the first experiment is to demonstrate the improved performance of our algorithm from Section \ref{sec:algorithm} in comparison to the classical filtered derivative approach in a stylized problem setup.  Recall that the filtered derivative method is equivalent to omitting the proximal denoising step in our algorithm, i.e., $\hat{\bfx}[t] = \bar{\bfy}[t]$ in Step $3$ of our algorithm.


We consider a signal sequence $\bfx^{\star}[t] \in \mathbb{R}^{200 \times 200}, t = 1,\ldots,100,$ consisting of exactly one change-point at time $t=50$. Let $\textbf{u}^{(1)},\textbf{u}^{(2)},\textbf{v}^{(1)},\textbf{v}^{(2)} \in \R^{200}$  be vectors with unit Euclidean-norm and direction chosen uniformly at random and independently. The signal $\bfx^{\star}[t]$ is a $200 \times 200$ matrix equal to $\textbf{u}^{(1)}\textbf{v}^{(1)\prime}$ before the change-point and $\textbf{u}^{(2)}\textbf{v}^{(2)\prime}$ after the change-point, and the observations are $\bfy[t] = \bfx^{\star}[t] + \bfe[t], t=1,\ldots,100$, where $\bfe[t] \sim \mathcal{N}(0,\sigma^2 I_{{200^2 \times 200^2}})$ with $\sigma = 0.04$.  Given this sequence of observations, we apply our algorithm with parameters $\lambda = 0.4$ and $\theta = 5$ (and with proximal denoising based on the nuclear-norm), and the filtered derivative algorithm with $\theta = 5$.  The corresponding derivative values from our algorithm and the filtered derivative algorithm are given in the left sub-plot of Figure~\ref{fig:PlantedLRcombine}.  We repeat the same experiment with the modification that the vectors $\textbf{u}^{(1)},\textbf{u}^{(2)},\textbf{v}^{(1)},\textbf{v}^{(2)}$ now have Euclidean-norm equal to $2$, thus leading to a larger-sized change relative to the noise.  The corresponding derivative values from our algorithm and the filtered derivative algorithm are given in the right sub-plot in Figure~\ref{fig:PlantedLRcombine}.


One observation is that the derivative values are generally larger with the filtered derivative algorithm than with our approach; this is primarily due to the lack of a denoising step, as a larger amount of noise leads to greater derivative values.  More crucially, however, the \emph{relative difference} in the derivative values near a change-point and away from a change-point is much larger with our algorithm than with the filtered derivative method.  This is also a consequence of the inclusion of the proximal denoising step in our algorithm and the lack of a similar denoising operation in the filtered derivative approach.  By suppressing the impact of noise via proximal denoising, our approach identifies smaller-sized changes more reliably than a standard filtered derivative method without a denoising step (see the sub-plot on the left in Figure~\ref{fig:PlantedLRcombine}).

%

\textbf{Estimating change-points in sequences of sparse vectors.} In our second experiment, we investigate the variation in the performance of our algorithm by choosing different sets of parameters $\theta,\lambda,\gamma$.  We consider a signal sequence $\bfx^\star[t] \in \R^{1000}, ~ t=1,\dots,1000$ consisting of sparse vectors.  Specifically, we begin by generating $10$ sparse vectors $\mathbf{S}^{(k)} \in \R^{1000}, ~ k=1,\dots,10$ as follows: for each $k=1,\dots,10$, the vector $\mathbf{S}^{(k)}$ is a random sparse vector consisting of $30$ nonzero entries (the locations of these entries are chosen uniformly at random and independently of $k$), with each nonzero entry being set to $1.2^{k-1}$.  We obtain the signal sequence $\bfx^\star[t]$ from the $\mathbf{S}^{(k)}$'s by setting $\bfx^\star[t] = \mathbf{S}^{(k)}$ for $t \in \{100(k-1)+1,100k\}$.  In words, the signal sequence $\bfx^\star[t]$ consists of $10$ equally-sized blocks of length $100$, and within each block the signal is identically equal to a sparse vector consisting of $30$ nonzero entries.  The magnitudes of the nonzero entries of $\bfx^\star[t]$ in the latter blocks are larger than those in the earlier blocks.  The observations are $\bfy[t] = \bfx^\star[t] + \bfe[t], ~ t=1,\dots,1000$, where each $\bfe[t] \sim \mathcal{N}(0,\sigma^2 I_{1000\times 1000})$ with $\sigma$ chosen to be $2.5$.  We then apply our proposed algorithm using the four choices of parameters listed in Figure~\ref{fig:sv_tableparam}, with a proximal operator based on the $\ell_1$-norm. The estimated sets of change-points are given in Figure~\ref{fig:sv_cpestimates}, and the derivative values corresponding to Step 4 of our algorithm are given in Figure~\ref{fig:sv_antd}.

\begin{figure}
\centering
\begin{tabular}{| c | c c c| }
\hline
Run & $\theta$ & $\lambda$ & $\gamma$ \\
\hline
1 & 10 & 1 & 15 \\
2 & 10 & 2 & 9 \\
3 & 30 & 1 & 8 \\
4 & 30 & 2 & 6 \\
\hline
\end{tabular}
\caption{Table of parameters employed in our change-point estimation algorithm in synthetic experiment with sparse vectors.}
\label{fig:sv_tableparam}
\end{figure}

\begin{figure}
\centering
\includegraphics[width=0.5\textwidth]{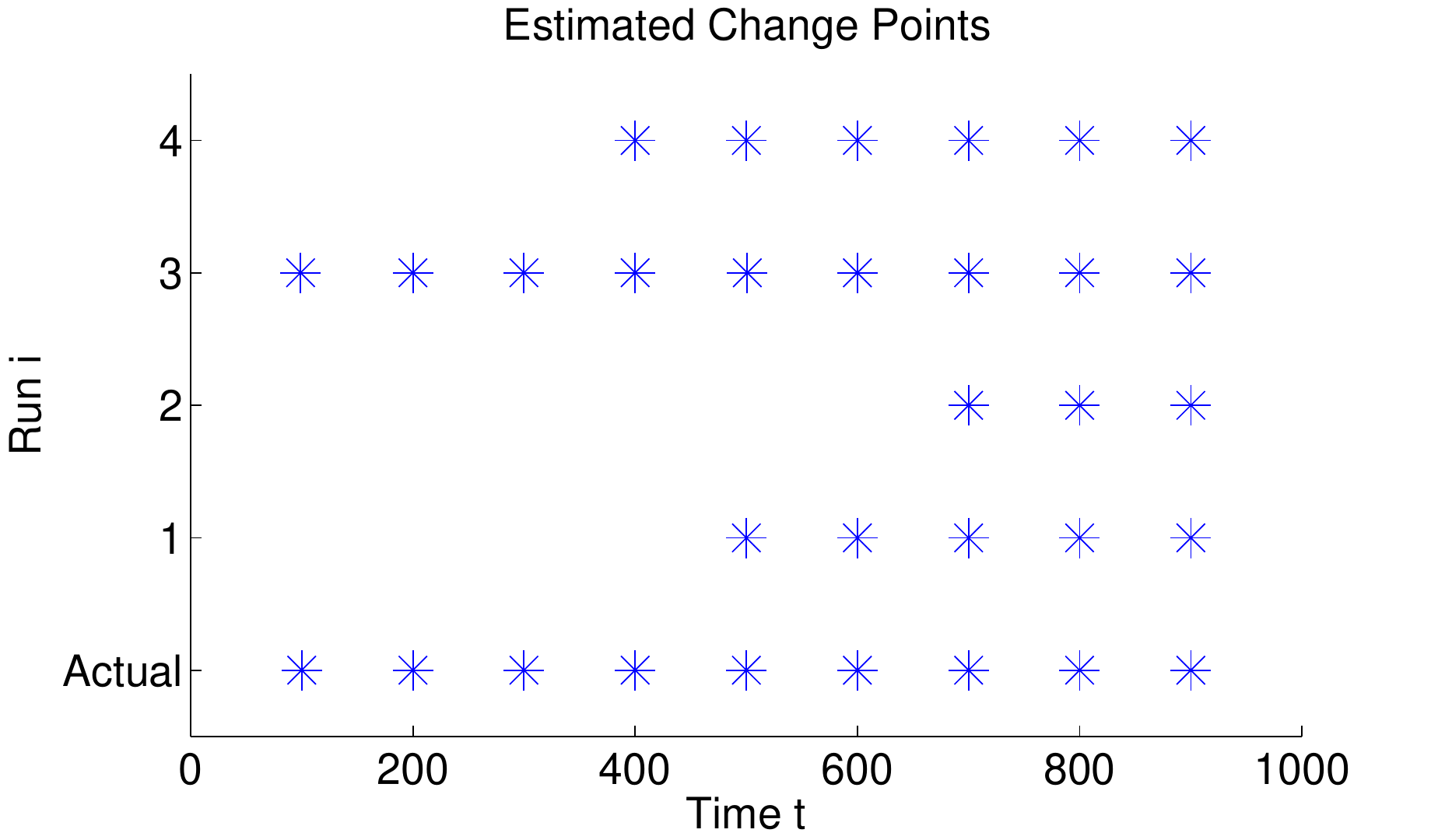}
\caption{Plot of estimated change-points: the locations of the actual change-points are indicated in the bottom row.}
\label{fig:sv_cpestimates}
\end{figure}

\begin{figure}
\centering
\includegraphics[width=0.5\textwidth]{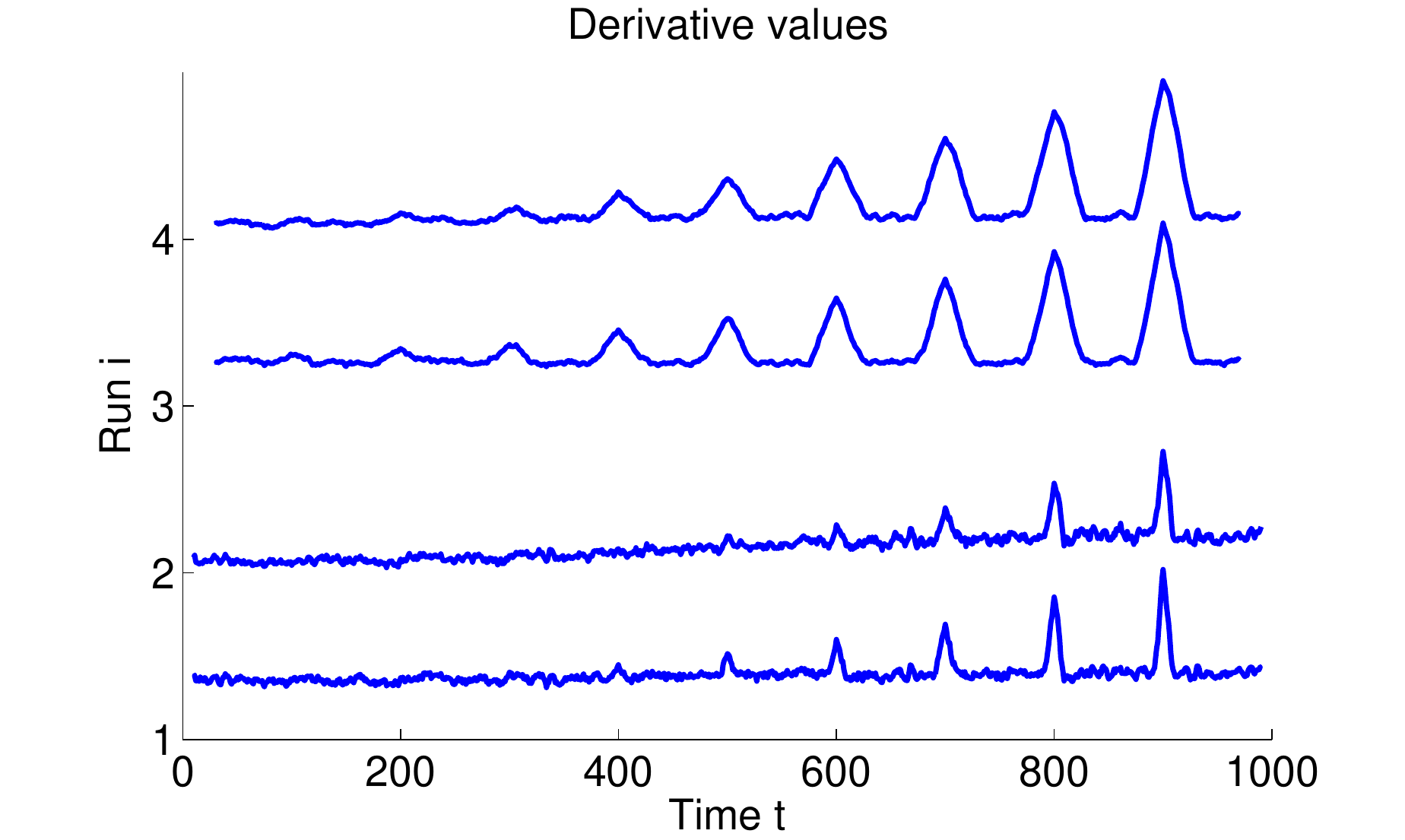}
\caption{Experiment with sparse vectors: graphs of derivative values corresponding to different parameters choices from Figure~\ref{fig:sv_tableparam}.}
\label{fig:sv_antd}
\end{figure}


First, note that the algorithm generally detects smaller sized changes with larger values of $\theta$ and smaller values of $\gamma$ (corresponding to Runs $3$ and $4$ from Figure~\ref{fig:sv_tableparam}), i.e., the averaging window size is larger in Step $2$ of our algorithm and the threshold is smaller in Step $4$.  Next, consider the graph of derivative values in Figure~\ref{fig:sv_antd}.  The estimated locations of change-points correspond to peaks in Figure~\ref{fig:sv_antd}, so the algorithm can be interpreted as selecting a subset of peaks that are sufficiently separated (Step $6$).  We note that a smaller choice of $\theta$ leads to sharper peaks (and hence, smaller-sized groups in Step $6$), while a larger choice of $\theta$ leads to wider peaks (correspondingly, larger-size groups in Step $6$).

\section{Conclusions} \label{sec:conclusions}

We propose an algorithm for high-dimensional change-point estimation that blends the filtered derivative method with a convex optimization step that exploits low-dimensional structure in the underlying signal sequence.  We prove that our algorithm reliably estimates change-points provided the product of the square of the size of the smallest change (measured in $\ell_2$-norm) and the smallest distance between changes is larger than Gaussian distance/width quantity $\eta^2$, which characterizes the low-dimensional complexity in the signal sequence.  The dependence on $\eta^2$ is a result of the integration of the convex optimization step (based on proximal denoising).

The change-point literature also consists of extensive investigations of \emph{quickest change detection} problems \cite{Lor:71,Pag:54,PooHad:08,Shi:63,VeeBan:12}, which are qualitatively somewhat different than the setup considered in our work.  In those settings one is given access to observations sequentially, and the objective is to correctly declare when a change-point occurs in the shortest time possible (i.e., minimize the expected delay) subject to false alarm rate constraints.  Building on the algorithmic ideas described in this paper, it would be of interest to design computationally and statistically efficient techniques for high-dimensional quickest change detection problems by exploiting structure in the underlying signal sequence.

%
%

\section*{Acknowledgments} \label{sec:acknowledgments} This work was supported in part by the following sources: National Science Foundation Career award CCF-1350590, Air Force Office of Scientific Research grant FA9550-14-1-0098, an Okawa Research Grant in Information and Telecommunications, and an A*STAR (Agency for Science, Technology and Research, Singapore) Fellowship.  Yong Sheng Soh would like to thank Michael McCoy for useful discussions, and Atul Ingle for pointing out a typographical error in a preliminary version of this paper.



\bibliography{cpd}


\section{Appendix}

The Appendix is divided as follows.  In Section \ref{sec:gaussiandistwidth} we describe the relation between the Gaussian distance and the Gaussian width.  Next, in Section \ref{sec:prep} we analyze the denoising properties of proximal operators.  Finally, in Section \ref{sec:proofmainresult} we prove the main results (from Section \ref{sec:mainresults}) of this paper. As described at the end of Section \ref{sec:minkowskifunc}, we reiterate that the assumption that the set $\mathrm{conv}(\mathcal{A})\subset \mathbb{R}^p$ contains the origin in its interior holds throughout the Appendix.


\subsection{Relationship between Gaussian distance and Gaussian width} \label{sec:gaussiandistwidth}

The \emph{Gaussian width} of a set $S \subseteq \mathbb{R}^p$ is defined as \cite{Gor:88}:
\begin{equation*}
\omega(S) := \expc_{\mathbf{g}\sim \mathcal{N}(\mathbf{0},I_{p \times p})} \biggl[ \sup_{\mathbf{z}\in S} \mathbf{g}^{T} \mathbf{z} \biggr].
\end{equation*}

The next definition that we need in order to relate the Gaussian distance and the Gaussian width is the \emph{skewness} $\kappa_{\mathcal{C}}(\bfx)$ of a norm $\|\cdot \|_{\mathcal{C}}$ at a point $\bfx$:
\begin{equation*}
\kappa_{\mathcal{C}}(\bfx) :=\frac{ \|\Pi_{\partial \| \bfx \|_{\mathcal{C}}} (\mathbf{0})\|_{2} }{ \|\Pi_{\text{aff.hull.}(\partial \| \bfx \|_{\mathcal{C}})} (\mathbf{0}) \|_{2}},
\end{equation*}
where $\Pi$ denotes the Euclidean projection and $\text{aff.hull.}$ denotes the affine hull. The quantity $\kappa$ has a natural geometric interpretation: since the subdifferential $\partial \| \bfx \|_{\mathcal{C}}$ corresponds to a face of the dual norm ball $\mathcal{C}^{*} = \{ \bfx: \| \bfx\|_{\mathcal{C}}^{*} \leq 1\}$, the parameter $\kappa_{\mathcal{C}}(\bfx)$ measures the skewness of the face $\partial \| \bfx \|_{\mathcal{C}}$. It is clear from this interpretation that $\kappa_{\mathcal{C}}(\bfx)=1$ for all $\bfx \in \R^p$ whenever the unit ball with respect to the dual norm is suitably symmetric.  Examples of such convex sets include the $\ell_{1}$-norm ball, the nuclear-norm ball, the $\ell_{\infty}$-norm ball and the spectral-norm ball.

Our final definition relates to yet another convex-geometric concept.  The tangent cone $T_{\mathcal{C}}(\bfx)$ at a point $\bfx \in \R^p$  with respect to the unit ball of the $\|\cdot \|_{\mathcal{C}}$-norm (i.e., the convex set $\mathcal{C}$) when $\|\bfx\|_{\mathcal{C}} = 1$ is:
\begin{equation*}
T_{\mathcal{C}}(\bfx) := \mathrm{cone} \{ \mathbf{Z} - \bfx : \mathbf{Z} \in \mathbb{R}^p, \| \mathbf{Z}\|_{\mathcal{C}} \leq \| \bfx\|_{\mathcal{C}} \}.
\end{equation*}
For general unnormalized nonzero points $\bfx \in \R^p$, the tangent cone with respect to $\mathcal{C}$ is $T_{\mathcal{C}}(\bfx / \|\bfx\|_{\mathcal{C}})$.

The next proposition relates the Gaussian distance and the Gaussian width.  The result relaxes a ``weak decomposability'' assumption in \cite[Prop.1]{FoyMac:13}. We denote the Euclidean sphere in $\R^p$ by $\mathbb{S}^{p-1}$.

\begin{proposition} \label{thm:etabound}
The Gaussian distance is bounded above by the Gaussian width as follows:
\begin{equation*}
\eta_{\mathcal{C}}({\bfx}) \leq \omega(T_{\mathcal{C}}(\bfx) \cap \mathbb{S}^{p-1})+ 3\kappa_{\mathcal{C}}(\bfx) + 4.
\end{equation*}
%
%
\end{proposition}

Proposition \ref{thm:etabound} is useful because it relates Theorem \ref{thm:changepoint} and Proposition \ref{thm:squarederror} with previously computed bounds on Gaussian widths \cite{CRPW:12,FoyMac:13}.

The proof of Proposition \ref{thm:etabound} requires two short lemmas.

\begin{lemma} \label{lemma:w0exists}
Suppose $\bfx \neq 0$. Define $\mathcal{H}$ to be the affine hull of $\partial \| \bfx \|_{\mathcal{C}}$ and $\mathbf{w}_0 := \Pi_{\mathcal{H}}(\mathbf{0})$. Then $\langle \mathbf{w} - \mathbf{w}_0, \mathbf{w}_0 \rangle = 0$ for all $\mathbf{w} \in \partial \| \bfx \|_{\mathcal{C}}$ and $\| \mathbf{w}_0 \|_{2} > 0$.
\end{lemma}

\begin{proof}
Since $\bfx\neq \mathbf{0}$, the subdifferential $\partial \| \bfx \|_{\mathcal{C}}$ is a proper face of the dual norm ball $\mathcal{C}^{*} = \{ \bfy : \| \bfy \|_{\mathcal{C}}^{*}\leq1\} $. Also, since $\mathbf{w}_0 - \mathbf{0}$ is orthogonal to $\mathcal{H}$, we have $\langle \mathbf{w}_0 - \mathbf{0},\mathbf{w}- \mathbf{w}_0 \rangle = 0$ for all $\mathbf{w}\in \mathcal{H}$. In particular, this holds for all $\mathbf{w} \in \partial \| \bfx \|_{\mathcal{C}}$. By the assumption that the unit-norm ball has a non-empty interior (see reminder at the beginning of the Appendix), we have that $\mathbf{0} \in \mathrm{int}(\mathcal{C})$, which implies that $\mathbf{0} \in \mathrm{int}(\mathcal{C}^{*})$. Consequently, $\mathcal{H}$ does not contain $\mathbf{0}$ and thus $\mathbf{w}_0 \neq \mathbf{0}$. This implies that $\| \mathbf{w}_0 \|_{2} > 0$.
\end{proof}

\begin{lemma} \label{thm:lambdastarlipschitz} Let $\bfx \in \R^p$ be an arbitrary nonzero vector. Define $\lambda^{\star} : \mathbb{R}^{p} \mapsto \mathbb{R}$ as the function
\begin{align*}
\lambda^{\star} (\mathbf{g}) &:= \argmin_{\lambda \geq 0} \mathrm{dist} (\mathbf{g}, \lambda \cdot \partial \| \bfx \|_{\mathcal{C}}) \\
&= \argmin_{\lambda \geq 0} \mathrm{dist}^2 (\mathbf{g}, \lambda \cdot \partial \| \bfx \|_{\mathcal{C}}).
\end{align*}
Let $\mathcal{H}$ be the affine hull of $\lambda \cdot \partial \| \bfx \|_{\mathcal{C}}$. Then $\lambda^{\star}$ is $\frac{1}{\mathrm{dist}(\mathbf{0},\mathcal{H})}$-Lipschitz.
\end{lemma}

\begin{proof}
Let $\mathbf{g}_1,\mathbf{g}_2$ be arbitrary vectors in $\mathbb{R}^{p}$. Since $\|\bfx\|_{\mathcal{C}}< \infty$, the subdifferential $\partial \| \bfx\|_{\mathcal{C}}$ is a closed convex set \cite{Roc:70}. Hence we may let $\mathbf{w}_{\mathbf{g}_1}$ be the point in $\partial \| \bfx \|_{\mathcal{C}}$ such that $\| \mathbf{w}_{\mathbf{g}_1}-{\mathbf{g}_1}\|_{2}=\mathrm{dist} ({\mathbf{g}_1}, \lambda^{\star}({\mathbf{g}_1}) \cdot \partial \|\bfx\|_{\mathcal{C}} )$. Define $\mathbf{w}_{\mathbf{g}_2}$ similarly. Let $\mathbf{w}_{0} = \Pi_{\mathcal{H}}(0)$ so that
\begin{eqnarray} \label{eq:lemma3w0}
\| \lambda^{\star}(\mathbf{g}_2) \mathbf{w}_{\mathbf{g}_2} &-& \lambda^{\star}(\mathbf{g}_1) \mathbf{w}_{g_1} \|_{2} \nonumber \\ &=& \| (\lambda^{\star}(\mathbf{g}_2) - \lambda^{\star}(\mathbf{g}_1)) \mathbf{w}_0 + (\lambda^{\star}(\mathbf{g}_2)(\mathbf{w}_{\mathbf{g}_2} - \mathbf{w}_0)+ \lambda^{\star}(\mathbf{g}_1)(\mathbf{w}_0 - \mathbf{w}_{\mathbf{g}_1})) \|_{2} \nonumber\\
& \geq &  \langle (\lambda^{\star}(\mathbf{g}_2) - \lambda^{\star}(\mathbf{g}_1)) \mathbf{w}_0 + (\lambda^{\star}(\mathbf{g}_2)(\mathbf{w}_{\mathbf{g}_2} - \mathbf{w}_0)+ \lambda^{\star}(\mathbf{g}_1)(\mathbf{w}_0 - \mathbf{w}_{\mathbf{g}_1})) , \mathbf{w}_0\rangle \frac{1}{\| \mathbf{w}_0 \|_2} \nonumber\\
&=& \| \mathbf{w}_0 \|_2|\lambda^{\star}(\mathbf{g}_2) - \lambda^{\star}(\mathbf{g}_1) |,
\end{eqnarray}
where the last equality follows from Lemma \ref{lemma:w0exists}. Recall that projection onto a nonempty, closed convex set is nonexpansive, and thus we have $\| {\mathbf{g}_2} - {\mathbf{g}_1} \|_2 \geq \| \Pi_{\cup_{\lambda \geq 0} \{\lambda \cdot \partial \|\bfx\|_{\mathcal{C}}\}}(\mathbf{g}_2) -\Pi_{\cup_{\lambda \geq 0} \{\lambda \cdot \partial \|\bfx\|_{\mathcal{C}}\}} (\mathbf{g}_1)\|_2 =\| \Pi_{\lambda^{\star}({\mathbf{g}_2}) \cdot \partial \|\bfx\|_{\mathcal{C}}}(\mathbf{g}_2) -\Pi_{\lambda^{\star}({\mathbf{g}_1}) \cdot \partial \|\bfx\|_{\mathcal{C}}} (\mathbf{g}_1)\|_2 = \| \lambda^{\star}(\mathbf{g}_2) \mathbf{w}_{\mathbf{g}_2} - \lambda^{\star}(\mathbf{g}_1) \mathbf{w}_{\mathbf{g}_1} \|_2 \geq \| \mathbf{w}_0 \|_2|\lambda^{\star}(\mathbf{g}_2) - \lambda^{\star}(\mathbf{g}_1) | $.
\end{proof}

\emph{Proof of Proposition \ref{thm:etabound}}

Our proof is a minor modification of the proof of \cite[Prop.1.]{FoyMac:13}.  Let $\mathcal{H}$ be the affine hull of $\partial \| \bfx \|_{\mathcal{C}}$ and $\mathbf{w}_0 = \Pi_{\mathcal{H}}(\mathbf{0})$. From Lemma \ref{thm:lambdastarlipschitz}, we have $\lambda^{\star}$ is $\frac{1}{\| \mathbf{w}_0 \|_2}$-Lipschitz function. Hence by \cite[Theorem 5.3]{Led:01}, we have $|\lambda^{\star} (\bfe) - \expc[\lambda^{\star}(\bar{\bfe})] | \leq c $ for $\bar{\bfe} \sim \mathcal{N}(\mathbf{0}, I_{p \times p})$ with probability greater than $1-2\exp(-(c \| \mathbf{w}_0 \|_2)^2 /2)$.  Suppressing the dependence on $\bar{\bfe}$, consider the event $E_c :=\{|\lambda^{\star} (\bfe) - \expc[\lambda^{\star}] | \leq c  \}$, and condition on this event. Define $\mathbf{w}_1:=\Pi_{\partial \| \bfx \|_{\mathcal{C}}} (\mathbf{0})$ so that $\| \mathbf{w}_1\|_2 / \| \mathbf{w}_0 \|_2 = \kappa(\bfx)$. Let $\mathbf{w}_{\bfe} \in \partial \| \bfx \|_{\mathcal{C}}$ be such that $\| \mathbf{w}_{\bfe}-{\bfe}\|_{2}=\mathrm{dist} ({\bfe}, \lambda^{\star}({\bfe}) \cdot \partial \|\bfx\|_{\mathcal{C}} )$. One has that $\frac{\lambda^{\star} (\bfe) }{\expc [\lambda^{\star}]+ c} \mathbf{w}_{\bfe}+ \frac{\expc[\lambda^{\star}]+ c -\lambda^{\star}(\bfe)}{\expc[\lambda^{\star}]+c} \mathbf{w}_1$ is a convex combination of $\mathbf{w}_1$ and $\mathbf{w}_{\bfe}$ (as we condition on $E_c$), and hence it belongs to $\partial \| \bfx \|_{\mathcal{C}}$. Then
\begin{align} \label{eq:etabound_distbound}
\mathrm{dist}(\bfe, (\expc [\lambda^{\star}]+c) \cdot \partial \| \bfx \|_{\mathcal{C}}) &\overset{\text{(i)}}{\leq} \| \bfe  - ( \lambda^{\star}(\bfe)  \mathbf{w}_{\bfe} + (\mathbb{E}[\lambda^{\star}]+ c - \lambda^{\star}(\bfe)) \mathbf{w}_1 ) \|_2 \nonumber\\
&\overset{\text{(ii)}}{\leq} \mathrm{dist} (\bfe,\lambda^{\star}({\bfe}) \cdot \partial \|\bfx\|_{\mathcal{C}} )) + \|  (\mathbb{E}[\lambda^{\star}]+ c - \lambda^{\star}(\bfe)) \mathbf{w}_1 \|_{2} \nonumber\\
&\overset{\text{(iii)}}{\leq} \mathrm{dist} (\bfe,\lambda^{\star}({\bfe}) \cdot \partial \|\bfx\|_{\mathcal{C}} )) + 2c \kappa(\bfx) \|\mathbf{w}_0 \|_2
\end{align}
where (i) is a consequence of $\frac{\lambda^{\star} (\bfe) }{\expc [\lambda^{\star}]+ c} \mathbf{w}_{\bfe}+ \frac{\expc[\lambda^{\star}]+ c -\lambda^{\star}(\bfe)}{\expc[\lambda^{\star}]+c} \mathbf{w}_1 \in \partial \| \bfx \|_{\mathcal{C}}$, (ii) follows from the triangle inequality, and (iii) follows from the definition of $\kappa(\bfx)$ and our conditioning on the event $E_c$. Define the function $m: \mathbb{R}^{p} \mapsto \mathbb{R}$
\begin{equation*}
m(\bfe) = \mathrm{dist}(\bfe, (\expc [\lambda^{\star}]+c) \cdot \partial \| \bfx \|_{\mathcal{C}}) - \mathrm{dist}(\bfe,\lambda^{\star}({\bfe}) \cdot \partial \|\bfx\|_{\mathcal{C}}).
\end{equation*}
Since $m(\bfe)$ is the difference of two $1$-Lipschitz functions and hence $2$-Lipschitz, we have the concentration inequality $\prob(m < \expc [m] - r ) \leq \exp(-r^2 / 8)$. By setting $r =  \sqrt{8 \log(1/(1 - 2 \exp (- (c\|\mathbf{w}_0 \|_2)^2/2)))}$ we have $\exp(-r^2/8) = \allowbreak 1- \allowbreak 2\exp(-(c \| \mathbf{w}_0 \|_2)^2 /2)$. From \eqref{eq:etabound_distbound} the event $\{ m(\bfe) \leq 2c \kappa(\bfx) \|\mathbf{w}_0 \|_2 \} $ holds with probability greater than $1-2\exp(-(c \| \mathbf{w}_0 \|_2)^2 /2)$. Hence it must be the case that 
\begin{equation} \label{eq:propgaussianwidthintstep}
\expc[m(\bfe) ] \leq 2\kappa(\bfx) c \|\mathbf{w}_0\|_2 + \sqrt{8 \log(1/(1 - 2 \exp (- (c\|\mathbf{w}_0 \|)^2/2)))}.
\end{equation}

Define $N: = \cup_{\lambda \geq 0} \{ \lambda \cdot \partial \|\bfx \|_{\mathcal{C}}\} $. We have
\begin{eqnarray*}
 \eta_{\mathcal{C}} (\bfx) &=& \inf_{\lambda} \bigl\{ \expc [\mathrm{dist}(\bfe, \lambda \cdot \partial \| \bfx \|_\mathcal{C})]  \bigr\} \\
&\leq&  \expc [\mathrm{dist}(\bfe, (\expc[\lambda^{\star}]+c) \cdot \partial \| \bfx \|_\mathcal{C})]   \\
&=& \expc[\mathrm{dist}(\bfe,\lambda^{\star}({\bfe}) \cdot \partial \|\bfx\|_{\mathcal{C}})]+ \expc [m(\bfe)]  \\
&\overset{\text{(i)}}{=}& \expc[\mathrm{dist}(\bfe,N)]+ \expc [m(\bfe)]  \\
&\overset{\text{(ii)}}{\leq}& \expc [\mathrm{dist} (\bfe,N )]  + 2\kappa(\bfx) c \|\mathbf{w}_0\|_2 + \sqrt{8 \log(1/(1 - 2 \exp (- (c\|\mathbf{w}_0 \|)^2/2)))} \\
&\overset{\text{(iii)}}{\leq}& \{\expc [\mathrm{dist}^2 (\bfe,N )]\}^{1/2}  + 2\kappa(\bfx) c \|\mathbf{w}_0\|_2 + \sqrt{8 \log(1/(1 - 2 \exp (- (c\|\mathbf{w}_0 \|)^2/2)))} \\
&\overset{\text{(iv)}}{\leq}&  \omega(T_{\mathcal{C}}(\bfx) \cap \mathbb{S}^{p-1}) + 1+ 2\kappa(\bfx) c\| \mathbf{w}_0\|_2 +\sqrt{8 \log(1/(1 - 2 \exp (- (c \|\mathbf{w}_0 \|)^2/2)))} ,
\end{eqnarray*}
where (i) follows from the definition of $\lambda^{\star}$, (ii) follows from \eqref{eq:propgaussianwidthintstep}, (iii) follows Jensen's Inequality, and (iv) follows from \cite[Proposition 10.1]{ALMT:13}. We obtain the desired bound by setting $c=1.5 / \|\mathbf{w}_0\|_2$. $\qed$


%


\subsection{Analysis of proximal denoising operators} \label{sec:prep}


The first result describes a useful monotonicity property of convex functions \cite{Min:64}.

\begin{lemma}[Monotonicity, \cite{Min:64}] \label{thm:monotonicity} Let $f$ be a convex function. Let $\bfx_1,\bfx_2 \in \mathbb{R}^p$. Then for any $\mathbf{Z}_{i} \in \partial f(\bfx_i), i= 1,2$, we have
\begin{equation*}
\langle \mathbf{Z}_1 - \mathbf{Z}_2 , \bfx_1 - \bfx_2 \rangle \geq 0.
\end{equation*}
\end{lemma}

Our second result applies this monotonicity property to show that the error of proximal denoising operators is robust to small changes in the underlying signal $\bfx^{\star}$.  Notice that our proposition also describes the performance of proximal denoisers for \emph{combinations} of structured signals corrupted by noise.  This is relevant in our subsequent analysis because the proximal denoiser is applied to averages computed near change-points.
\begin{proposition}[Robustness] \label{thm:error_robust} Suppose $\hat{\bfx} = \argmin_{\bfx} \frac{1}{2} \| \bfx^{\star} + \bfe - \bfx \|^2_{2} + f(\bfx)$ for some convex function $f$ and

\noindent Case 1 (Convex combination of two structured signals): $\bfx^{\star} = \mu \bfx^{\star}_{0}+ (1-\mu) \bfx^{\star}_{1}$ for some $0\leq \mu \leq 1$ is a convex combination of two signals $\bfx^{\star}_{0}$ and $\bfx^{\star}_{1}$. Then
\begin{equation*}
\expc \bigl[ \| \bfx^{\star}_{0}-\hat{\bfx} \|_2 \bigr] \leq (1-\mu) \| \bfx^{\star}_{0} - \bfx^{\star}_{1} \|_2 +  \expc[\mathrm{dist}(\bfe,\partial f(\bfx^{\star}_{0} ) )].
\end{equation*}
In particular when $\mu = 1$ there is no mixture. In this special case the error bound simplifies to
\begin{equation*}
\expc \bigl[ \| \bfx^{\star}_{0}-\hat{\bfx} \|_2 \bigr] \leq  \expc[\mathrm{dist}(\bfe,\partial f(\bfx^{\star}_{0} ) )].
\end{equation*}

\noindent Case 2 (Small perturbation to a structured signal): $\bfx^{\star} = \bfx^{\star}_0 + \boldsymbol\Delta$. Then
\begin{equation*}
\expc \bigl[ \| \bfx^{\star}_{0}-\hat{\bfx} \|_2 \bigr] \leq \| \boldsymbol\Delta \|_2 +  \expc[\mathrm{dist}(\bfe,\partial f(\bfx^{\star}_{0} )  )].
\end{equation*}
Here the expectations are with respect to $\bfe$.
\end{proposition}

\begin{proof}
We only prove Case 1 since Case 2 follows from a change of variables. We begin by fixing an $\bfe$. From the optimality conditions, we have $\mu \bfx^{\star}_{0}+ (1-\mu) \bfx^{\star}_{1}+\bfe - \hat{\bfx}  \in \partial \| \hat{\bfx}\|_{\mathcal{C}}$. Let $\mathbf{Z}_0 = \argmin_{\mathbf{Z} \in \partial f(\bfx^{\star}_{0})} \| \mathbf{Z} -  \bfe\| $. From the monotonicity property in Lemma \ref{thm:monotonicity} we have
\begin{equation*}
\langle \mu \bfx^{\star}_{0}+ (1-\mu) \bfx^{\star}_{1}+\bfe - \hat{\bfx} -\mathbf{Z}_0,\hat{\bfx} - \bfx^{\star}_{0} \rangle \geq 0 .
\end{equation*}
Rearranging terms and applying the Cauchy-Schwarz inequality, we obtain
\begin{equation*}
(1-\mu) \|  \bfx^{\star}_{0} - \bfx^{\star}_{1}\|_2 \|  \bfx^{\star}_{0} - \hat{\bfx} \|_2 + \| \mathbf{Z}_0 - \bfe\|_2 \|  \bfx^{\star}_{0} - \hat{\bfx} \|_2 \geq \| \bfx^{\star}_{0}-\hat{\bfx} \|^2_2.
\end{equation*}
Finally, we divide through by $\| \bfx^{\star}_{0} -\hat{\bfx} \|_2$ and take expectations on both sides with respect to $\bfe$ to obtain the desired result.
\end{proof}

The final result concerns a Lipschitz property of proximal operators. Demonstrating such a property allows us to subsequently appeal to concentration of measure results \cite{Led:01}.


\begin{lemma} [Proximal operators are non-expansive, Section 5 of \cite{Mor:65}]\label{thm:conc_eq} Suppose $f$ is a convex function. Let $\hat{\bfx}(\bfe)$ be the optimal solution of the following optimization problem
\begin{equation} \label{eq:thm_conc_eq_opt}
\hat{\bfx}(\bfe) = \argmin_{\bfx} \frac{1}{2} \| \bfx^{\star}+\bfe - \bfx \|^2_2 + f(\bfx).
\end{equation}
Then $\|\bfe_1 - \bfe_2 \|_2 \geq \|\hat{\bfx}(\bfe_1) - \hat{\bfx}(\bfe_2) \|_2 $.
\end{lemma}

\begin{corollary} \label{thm:optlipschitz}
Fix an $\bfx^{\star} \in \mathbb{R}^p$. Define the function $h :\mathbb{R}^{p} \mapsto \mathbb{R}$ as $h(\bfe) = \| \hat{\bfx}(\bfe) - \bfx^{\star}\|_2$, where $\hat{\bfx}(\bfe)$ is defined in \eqref{eq:thm_conc_eq_opt}. Then the function $h$ is $1$-Lipschitz.
\end{corollary}

\begin{proof}
By applying the triangle inequality twice one has $\|\hat{\bfx}(\bfe_1) - \hat{\bfx}(\bfe_2) \|_2 \geq \big{|} \|\hat{\bfx}(\bfe_1) - \bfx^{\star} \|_2 - \| \bfx^{\star} - \hat{\bfx}(\bfe_2) \|_2\big{|}$. The result follows from an application of Lemma \ref{thm:conc_eq}.
\end{proof}




\subsection{Proofs of results from Section \ref{sec:mainresults}} \label{sec:proofmainresult}

In this section we prove Proposition \ref{thm:probboundsE1E2E3} (our precursor to Theorem \ref{thm:changepoint}) and Proposition \ref{thm:squarederror}. To simplify notation, we denote $\eta_{\mathcal{C}}(\mathcal{X})$ by $\eta$  in this section.  First we establish a tertiary result that is useful for obtaining a sharper bound on the accuracy of the locations of the estimated change-points.
\begin{proposition}\label{thm:beforeafterlipschitz} Fix an $\bfx^{\star} \in \mathbb{R}^p$. Let $\hat{\bfx}_0$ and $\hat{\bfx}_1$ be the optimal solutions to $\hat{\bfx} = \argmin_{\bfx} \frac{1}{2} \| \bfx^{\star}+\bfe - \bfx \|^2_2 + f(\bfx) $ for $\bfe = \bfe_0$ and $\bfe=\bfe_1$, respectively. Define the function $j: \mathbb{R}^p \times \mathbb{R}^p \mapsto \mathbb{R}$,  $j(\bfe_0,\bfe_1) := \| \hat{\bfx}_0-\hat{\bfx}_1 \|_2$. Then $j$ is $\sqrt{2}$-Lipschitz.
\end{proposition}

\begin{proof}
Let $\{\hat{\bfx}_0^1,\hat{\bfx}_1^1 \}$ and $\{\hat{\bfx}_0^2,\hat{\bfx}_1^2 \}$ be the optimal solutions corresponding to the two instantiations $( \bfe_{0}^{1},\bfe_{1}^1)$ and $( \bfe_{0}^{2},\bfe_{1}^2)$ of the vectors $( \bfe_{0},\bfe_{1})$. From Lemma \ref{thm:conc_eq}, we have $\| \hat{\bfx}_{0}^{1} - \hat{\bfx}_0^2 \|_2 \leq \| \bfe_0^1  - \bfe_0^2 \|_2$ and  $\| \hat{\bfx}_1^1 - \hat{\bfx}_1^2 \|_2 \leq \| \bfe_1^1  - \bfe_1^2 \|_2$.  By applying the triangle inequality, we have $\| \hat{\bfx}_0^1 - \hat{\bfx}_1^1 \|_2 \leq \| \hat{\bfx}_0^1 - \hat{\bfx}_0^2 \|_2 + \| \hat{\bfx}_0^2 - \hat{\bfx}_1^2 \|_{2} + \| \hat{\bfx}_1^2 - \hat{\bfx}_1^1 \|_2$. Then
\begin{align*}
 | \| \hat{\bfx}_0^1 - \hat{\bfx}_1^1 \|_2-  \| \hat{\bfx}_0^2 - \hat{\bfx}_1^2 \|_{2} | &\leq \| \hat{\bfx}_0^1 - \hat{\bfx}_0^2 \|_2 + \| \hat{\bfx}_1^2 - \hat{\bfx}_1^1 \|_2 \\
&\leq \| \bfe_0^1  - \bfe_0^2 \|_2 + \| \bfe_1^1  - \bfe_1^2 \|_2 \\
&\leq \sqrt{2} \| (\bfe_0^1 ,\bfe_1^1)- (\bfe_0^2,\bfe_1^2) \|_{2}.
\end{align*}
Hence, $j$ is $\sqrt{2}$-Lipschitz.
\end{proof}

\emph{Proof of Proposition~\ref{thm:probboundsE1E2E3}} We divide the proof into three parts corresponding to the three events of interest.

\emph{Part one} [$\prob (E_{1}^c ) \leq 2n^{1-r^2}$]: For each change-point $t \in \tau^{\star}$, define the following event $ E_{1,t} : \{ S_t \geq \gamma\}$. Clearly, $E_1^{c} = \bigcup_{t \in \tau^{\star}}  E_{1,t}^{c}$. We will prove that $\prob (E_{1,t}^c ) \leq 2n^{-r^2}$. By taking a union bound over all $t \in \tau^{\star}$, we have
\begin{equation*}
\prob(E_1^c) = \prob(\bigcup_{t \in \tau^{\star}}  E_{1,t}^{c}) \leq  \sum_{t \in \tau^{\star}} \prob( E_{1,t}^{c}) \leq  2 |\tau^{\star}|n^{-r^2} \leq 2 n^{1-r^2}.
\end{equation*}
We now prove that  $\prob (E_{1,t}^c ) \leq 2n^{-r^2}$. Conditioning on the event $E^c_{1,t}$, and by the triangle inequality, we have
\begin{align*}
\gamma >& \| \hat{\bfx}[t - \theta +1] - \hat{\bfx}[t +1]\|_2 \\
\geq& - \| \bfx^{\star}[ t - \theta +1] - \hat{\bfx}[t - \theta +1] \|_2 + \| \bfx^{\star}[t  +1]  -\bfx^{\star}[t - \theta +1] \|_2 - \| \hat{\bfx}[t +1]-\bfx^{\star}[t  +1] \|_2.
\end{align*}
Since $ \| \bfx^{\star}[t  +1]-\bfx^{\star}[t - \theta +1] \|_2 \geq \Delta_{\min}\geq 2\gamma$, one of the two events $\{\| \hat{\bfx}[t-\theta+1] - \bfx^{\star}[t] \|_2 \geq \gamma/2\} $ or $\{\| \hat{\bfx}[t+1] - \bfx^{\star}[t+1] \|_2 \geq \gamma/2\}$ must occur. Also, since $t \in \tau^{\star}$, we have $|t-t'| \geq \theta$ for all $t' \in \tau^{\star} \backslash \{ t\}$. Hence the signal is constant over the time instances $\{t - \theta +1,\dots, t \}$ and $\{t +1,\dots,t+\theta\} $. By applying Proposition~\ref{thm:error_robust}, we have the inequalities $ \expc [\| \bfx^{\star}[t - \theta +1] - \hat{\bfx}[ t - \theta +1] \|_2]\leq \frac{\sigma}{\sqrt{\theta}}\eta$ and $\expc[ \| \hat{\bfx}[t +1]-\bfx^{\star}[t  +1] \|_2] \leq \frac{\sigma}{\sqrt{\theta}}\eta$. Thus
\begin{eqnarray*}
\prob (E_{1,t}^c )
&\leq& \prob(\| \hat{\bfx}[t-\theta+1] - \bfx^{\star}[t] \|_2 \geq \gamma/2)  + \prob(\| \hat{\bfx}[t+1] - \bfx^{\star}[t+1] \|_2 \geq \gamma/2) \\
&\leq& \prob(\| \hat{\bfx}[t-\theta+1] - \bfx^{\star}[t] \|_2 \geq \expc[\| \hat{\bfx}[t-\theta+1] - \bfx^{\star}[t] \|_2] + r\sqrt{\sigma^2 / \theta}\sqrt{2\log n}) \\
&&+ \prob( \| \hat{\bfx}[t+1] - \bfx^{\star}[t+1] \|_2 \geq \expc[\| \hat{\bfx}[t+1] - \bfx^{\star}[t+1] \|_2]+ r\sqrt{\sigma^2 / \theta}\sqrt{2\log n} ) \\
&\leq& 2\exp(-(r \sqrt{2\log n})^2 /2 ) = 2n^{-r^2}
\end{eqnarray*}
where the last inequality follows from Corollary \ref{thm:optlipschitz} and from \cite[Theorem 5.3]{Led:01}.

\emph{Part two} [$\prob(E_2^{c})\leq 2n^{1-r^2}$]:
We prove that $\prob(E_2^{c})\leq 2n^{1-r^2}$ in essentially the same manner in which we showed that $\prob(E_1^{c} )\leq 2n^{1-r^2}$. For all $t \in \tau_{\text{far}}$, define $E_{2,t}$ as the event $ E_{2,t}:=\{ \| \hat{\bfx}[t-\theta+1]- \hat{\bfx}[t+1] \|_2 \leq  \gamma \}$. Then $E_2^c = \bigcup_{t \in \tau_{\text{far}}} E_{2,t}^c $. We will start by proving that $\prob(E_{2,t}^c)\leq 2n^{-r^2} $.

By applying the triangle inequality and conditioning on the event $E_{2,t}^c $ holding for some $t \in \tau_{\text{far}}$, we have $\| \hat{\bfx}[t-\theta+1] - \bfx^{\star}[t+1] \|_2+ \| \bfx^{\star}[t+1]-\hat{\bfx}[t+1] \|_2  > \| \hat{\bfx}[t-\theta+1]- \hat{\bfx}[t+1] \|_2 > \gamma$.
Consequently, one of the two events $\{\| \hat{\bfx}[t-\theta+1] - \bfx^{\star}[t+1] \|_2 \geq \gamma/2 \}$ or $\{\| \bfx^{\star}[t+1]-\hat{\bfx}[t+1] \|_2 \geq \gamma/2\} $ must hold. Since $t \in \tau_{\text{far}}$, we have $|t-t^{\star} |>\theta$ for all $t^{\star} \in \tau^{\star}$, and thus the signal is constant over the time instances $\{t-\theta+1 , \dots,t+\theta\}$. By Proposition \ref{thm:error_robust}, we have $\expc[\| \hat{\bfx}[t-\theta+1] - \bfx^{\star}[t-\theta+1] \|_2] \leq \frac{\sigma}{\sqrt{\theta}}\eta$ and $\expc[\| \hat{\bfx}[t+1] - \bfx^{\star}[t+1] \|_2] \leq \frac{\sigma}{\sqrt{\theta}}\eta$. This implies that we have that at least one of the following two events $\{\| \hat{\bfx}[t-\theta+1] - \bfx^{\star}[t-\theta+1] \|_2  \geq \expc[\| \hat{\bfx}[t-\theta+1] - \bfx^{\star}[t-\theta+1] \|_2 ]+  r\sqrt{\sigma^2 / \theta}\sqrt{2\log n}\}$ or $\{\| \hat{\bfx}[t+1] - \bfx^{\star}[t+1] \|_2  \geq \expc[\| \hat{\bfx}[t+1] - \bfx^{\star}[t+1] \|_2 ]+   r\sqrt{\sigma^2 / \theta}\sqrt{2\log n}\}$ holds.

From Corollary \ref{thm:optlipschitz} and from \cite[Theorem 5.3]{Led:01}, we have that the probability of either event (corresponding to these two inequalities) occuring is less than $2\exp(-(r \sqrt{2\log n})^2 /2 ) = 2n^{-r^2}$. Thus
\begin{equation*}
\prob(E_2^c) = \prob(\bigcup_{t \in \tau_{\text{far}}}  E_{2,t}^{c}) \leq  \sum_{t \in \tau_{\text{far}}} \prob( E_{2,t}^{c}) \leq 2 |\tau_{\text{far}}|n^{-r^2} \leq  2n^{1-r^2},
\end{equation*}
as required.


\emph{Part three} [$\prob(E^{c}_{3}) \leq n^{1-r^2} $]: Let us now consider the event $E_3$. To simplify notation, we define $l:=4r\sqrt{\log n} / \eta$. To prove this part of the proposition, we show a slightly stronger result $\prob(E^{c}_{3}) \leq 4 \theta |\tau^{\star}|  \exp (- l^2 \eta^2/ 16 )$. Since $\theta |\tau^{\star}| \leq n/4$, our bound would imply that $\prob(E^{c}_{3}) \leq n^{1-r^2}$.

For all pairs $(t,\delta) \in \tau_{\text{buffer}}$, define the event $  E_{3,t,\delta } = \bigl\{ \| \hat{\bfx}[t +1] - \hat{\bfx}[t -\theta +1] \|_2 > \| \hat{\bfx}[t +1 +\delta] - \hat{\bfx}[t -\theta +1 +\delta] \|_2  \bigr\}$.
Then $E_3^c = \bigcup_{(t,\delta)\in \tau_{\text{buffer}}}  E_{3,t,\delta }^c$. We start by proving the following bound
\begin{equation*} 
\prob ( E^{c}_{3,t,\delta } )\leq 2\exp (- l^2 \eta^2/ 16 )
\end{equation*}
for all pairs $(t,\delta)$ in $\tau_{\text{buffer}}$. Fix one such pair and let $\Delta_t$ denote the magnitude of the change at $t \in \tau^{\star}$. From the triangle inequality and Proposition \ref{thm:error_robust} we have that
\begin{align*} 
\expc[\| \hat{\bfx}[t+1] -& \hat{\bfx}[t -\theta +1] \|_2 ]  \\
\geq& - \expc[ \| \hat{\bfx}[t +1] - \bfx^{\star}[t+1] \|_2 ]  \\ &+ \expc[\|\bfx^{\star}[t+1]  - \bfx^{\star}[t- \theta +1] \|_2] - \expc[\| \bfx^{\star}[t - \theta +1] - \hat{\bfx}[t- \theta +1] \|_2]  \\
\geq & ~ \Delta_t - 2 \sqrt{\sigma^2/\theta} \eta.
\end{align*}
Suppose that $\delta \geq 0$. By similarly applying the triangle inequality and Proposition \ref{thm:error_robust} we have
\begin{align*} 
\expc[\| \hat{\bfx}[t +1 +\delta] -& \hat{\bfx}[t -\theta +1 +\delta] \|_2] \\
\leq& ~ \expc[\| \hat{\bfx}[t +1 +\delta] - \bfx^{\star}[t +1] \|_2] + \expc[\| \bfx^{\star}[t +1] - \hat{\bfx}[t -\theta +1 +\delta] \|_2] \\
\leq& ~ (1-\delta / \theta) \Delta_t + 2 \sqrt{\sigma^2/\theta} \eta.
\end{align*} A similar set of computations will show that $ \expc[\| \hat{\bfx}[t +1 +\delta] - \hat{\bfx}[t -\theta +1 +\delta] \|_2] \leq (1+\delta / \theta) \Delta_t + 2 \sqrt{\sigma^2/\theta} \eta$ for $\delta<0$.
Combining these inequalities and using the range of values of $\delta$ we have
\begin{align} \label{eq:sharpness_bound}
\expc[ \| \hat{\bfx}[t+1] - \hat{\bfx}[t -\theta +1] \|_2] - \expc[ \| \hat{\bfx}[t +1 +\delta] - \hat{\bfx}[t -\theta +1 +\delta] \|_2] \geq& \frac{|\delta|}{\theta} \Delta_t - 4\frac{\sigma}{\sqrt{\theta}} \eta \nonumber \\
 \geq& l \frac{\sigma}{\sqrt{\theta}} \eta.
\end{align}
Then
\begin{align*}
\prob(E^{c}_{3,t,\delta }) =& \prob(  \| \hat{\bfx}[t+1 +\delta] - \hat{\bfx}[t -\theta +1 +\delta] \|_2  > \| \hat{\bfx}[t +1] - \hat{\bfx}[t -\theta +1] \|_2 \bigr ) \\
\overset{\text{(i)}}{\leq}& ~ \prob\Big{(} \| \hat{\bfx}[t+1 +\delta] - \hat{\bfx}[t -\theta +1 +\delta] \|_2  - \| \hat{\bfx}[t +1] - \hat{\bfx}[t -\theta +1] \|_2 \nonumber\\
& +\expc[ \| \hat{\bfx}[t+1] - \hat{\bfx}[t -\theta +1] \|_2] - \expc[ \| \hat{\bfx}[t +1 +\delta] - \hat{\bfx}[t -\theta +1 +\delta] \|_2]\geq  \frac{l\sigma}{\sqrt{\theta}} \eta \Bigg{)} \\
\overset{\text{(ii)}}{\leq}& ~ \prob\left(\expc[\| \hat{\bfx}[t+1] - \hat{\bfx}[t-\theta +1] \|_2] - \| \hat{\bfx}[t+1] - \hat{\bfx}[t -\theta +1] \|_2\geq \frac{l\sigma}{2\sqrt{\theta}} \eta \right) \\
& + \prob\Bigg{(}\| \hat{\bfx}[t+1 +\delta] - \hat{\bfx}[t -\theta +1 +\delta] \|_2 - \expc[\| \hat{\bfx}[t+1 +\delta] - \hat{\bfx}[t -\theta +1 +\delta] \|_2] \nonumber \\& \hspace{0.5in} \geq  \frac{l\sigma}{2\sqrt{\theta}} \eta \Bigg{)} \\
\overset{\text{(iii)}}{\leq}& 2\exp (- l^2 \eta^2/16 ),
\end{align*}
where (i) follows from \eqref{eq:sharpness_bound}, (ii) follows from the triangle inequality, and (iii) follows from Proposition \ref{thm:beforeafterlipschitz} and from \cite[Theorem 5.3]{Led:01}. Since $E_3^c = \bigcup_{(t,\delta)\in \tau_{\text{buffer}}}  E_{3,t,\delta }^c$, we have via a union bound
\begin{align*}
& \prob(E_3^c) \leq \sum_{(t,\delta)\in \tau_{\text{buffer}}} \prob ( E^{c}_{3,t,\delta } ) \leq 2|\tau_{\text{buffer}}| \exp (- l^2 \eta^2/16 )
\leq 4 \theta |\tau^{\star}|  \exp (- l^2 \eta^2/16 ).
\end{align*}
This concludes the proof of Proposition~\ref{thm:probboundsE1E2E3}. $\qed$


Before proving Proposition \ref{thm:squarederror} we require a short lemma.

\begin{lemma} \label{thm:errorboundwhp}
Let $\bfe \sim \mathcal{N} (0, \sigma^2 I_{p\times p})$. Then
\begin{equation*}
\mathrm{dist}^{2} ( \bfe, \lambda \cdot \partial \| \bfx \|_{\mathcal{C}}) \leq 2 \bigl( \expc \bigl[ \mathrm{dist} (\bfe,\lambda \cdot \partial\| \bfx \|_{\mathcal{C}}) \bigr] \bigr)^2 + 2\sigma^2 t^2
\end{equation*}
with probability greater than $1-2\exp(-t^2 / 2)$.
\end{lemma}

\begin{proof}
The mapping $\bfe \mapsto \mathrm{dist}(\bfe,\lambda \cdot \partial \| \bfx\|_{\mathcal{C}})$ is nonexpansive and hence $1$-Lipschitz. Using Theorem 5.3 from \cite{Led:01}, we have
\begin{equation} \label{eq:errorbounds_partabound}
\mathrm{dist}(\bfe,\lambda \cdot \partial \| \bfx \|_{\mathcal{C}}) \leq \expc \bigl[ \mathrm{dist}(\bfe,\lambda \cdot \partial \| \bfx \|_{\mathcal{C}}) \bigr] + t\sigma
\end{equation}
with probability greater than $1-\exp(-t^2 /2)$. By conditioning on the event corresponding to the inequality \eqref{eq:errorbounds_partabound}, we apply the arithmetic-geometric-mean inequality and conclude that
\begin{equation*}
\mathrm{dist}^2(\bfe,\lambda \cdot \partial \| \bfx \|_{\mathcal{C}}) \leq 2( \expc \bigl[ \mathrm{dist}(\bfe,\lambda \cdot \partial \| \bfx \|_{\mathcal{C}}) \bigr] )^2 + 2t^2\sigma^2
\end{equation*}
with probability greater than $1-\exp(-t^2 /2)$.
\end{proof}

\emph{Proof of Proposition \ref{thm:squarederror}}
It follows from the proof of Proposition \ref{thm:probboundsE1E2E3} that the event $E_1 \cap E_2$ holds with probability greater than $1-4n^{1-r^2}$. Conditioning on the event that $E_1 \cap E_2$ holds, the reconstructed signal is constant over the interval $\{t_1+\theta,\dots,t_2-\theta\}$. The result then follows from an application of Lemma \ref{thm:errorboundwhp} and a union bound. \qed

\end{document}